\documentclass[11pt]{amsart}

\usepackage[margin=1in]{geometry}
\usepackage{amsmath,amssymb,amsthm,mathtools}
\usepackage{xcolor}
\usepackage[colorlinks=true,linkcolor=blue!55!black,citecolor=blue!55!black,urlcolor=blue!55!black]{hyperref}
\usepackage{microtype}
\newtheorem{theorem}{Theorem}[section]
\newtheorem{proposition}[theorem]{Proposition}
\newtheorem{corollary}[theorem]{Corollary}
\newtheorem{lemma}[theorem]{Lemma}
\theoremstyle{definition}

\newtheorem{example}[theorem]{Example}
\theoremstyle{remark}
\newtheorem{remark}[theorem]{Remark}

\newcommand{\doi}[1]{\href{https://doi.org/#1}{doi:\nolinkurl{#1}}}
\newcommand{\K}{\mathbb K}
\newcommand{\SR}[1]{\K[#1]}
\DeclareMathOperator{\pdim}{pdim}
\DeclareMathOperator{\reg}{reg}
\DeclareMathOperator{\indim}{indim}

\title[Betti Numbers for Skeletons of Chordal Clique Complexes]{Explicit Betti Numbers for Skeletons of Chordal Clique Complexes and Their Alexander Duals}

\author{Mohammed Rafiq Namiq}
\address{Department of Mathematics, College of Science, University of Sulaimani, Sulaymaniyah, Kurdistan Region, Iraq}
\email{mohammed.namiq@univsul.edu.iq}


\begin{document}
	\begin{abstract}
		Let $\Delta_{\mathbf r}=\Delta_{\mathbf r}(n_1,\ldots,n_e)$ be the clique complex of a chordal graph with maximal cliques $S_1,\ldots,S_e$ in a leaf order, where $n_m=|S_m|$, $r_m=|S_{m+1}\cap(S_1\cup\cdots\cup S_m)|$, $\mathbf r=(r_1,\ldots,r_{e-1})$, and $N_{\mathbf r}$ is the number of vertices. We determine the graded Betti numbers $\beta_{i,j}\bigl(\mathbb K[(\Delta_{\mathbf r})_{(k)}]\bigr)$ of every skeleton and derive explicit formulas for its regularity, projective dimension, depth, multiplicity, Cohen--Macaulayness, and initially Cohen--Macaulayness. We also compute its extremal Betti numbers and describe its integral homology and homotopy type. For the Alexander dual $\Delta_{\mathbf r}^{\vee}$, we construct an explicit minimal multigraded resolution whose shifts are determined by $N_{\mathbf r}-n_m$ and $N_{\mathbf r}-r_m$, and obtain its graded and multigraded Betti numbers, canonical module, Cohen--Macaulay type, $a$-invariant, and Gorenstein criterion. We further show that every proper skeleton $(\Delta_{\mathbf r}^{\vee})_{(k)}$ is Cohen--Macaulay and level, determine its Betti numbers and type, and characterize its Gorenstein cases. Finally, we show that the graded Betti table of $\mathbb K[\Delta_{\mathbf r}^{\vee}]$ determines the multisets ${n_m}$ and ${r_m}$ up to a common shift.
	\end{abstract}
	
	\maketitle
	
	\section{Introduction}\label{sec:1}
	Stanley--Reisner theory relates the combinatorics of a simplicial complex to the homological properties of its face ring. We study this relation for clique complexes of chordal graphs, their skeletons, and their Alexander duals. If $\Sigma$ is a simplicial complex on a finite vertex set $X$, we write $\SR{\Sigma}=R/I_\Sigma$, where $R=\K[x:x\in X]$. Our main invariants are graded and multigraded Betti numbers, regularity, projective dimension, depth, multiplicity, Cohen--Macaulay type, and the topological information reflected in the resolutions.
	
	The clique complexes of chordal graphs are precisely the quasi-forests; see \cite{HHMTZ08,HHZ04}. Fr\"oberg uses the term fat forests for the same class \cite{Fr1}. Since these complexes are flag, Fr\"oberg's characterization of edge ideals with linear resolutions implies that their Stanley--Reisner ideals have $2$-linear resolutions \cite{Fr}. For fat forests, Fr\"oberg computed the Hilbert series and depth of the full Stanley--Reisner ring and, in the uniform case, the graded Betti numbers of the Alexander dual \cite{Fr1}. He later studied skeletons of thick trees, where each new simplex is attached along a single vertex \cite{Fr2}. In the present paper, the successive intersections are clique separators of arbitrary cardinality, including empty separators.
	
	Betti numbers of skeletons have also been studied without a chordality assumption. Roksvold and Verdure related the Betti numbers of a codimension one skeleton to those of the original complex \cite{RV15}, and formulas for arbitrary skeletons were given in \cite{Namiq2}. Proposition~\ref{prop:1} gives the part of the general skeleton relation needed here, with a direct proof from Hochster's formula and the Hilbert numerator. We combine that relation with the face-number formula from the leaf order in Proposition~\ref{prop:2}. For levelness, we use the known fact that the codimension one skeleton of a Cohen--Macaulay simplicial complex is level, due to Hibi, together with its iterated proper skeleton consequence; see \cite{HerzogHibi99,HibiLevel}. The level property, together with the explicit type and last shift formulas obtained below, yields precise Gorenstein criteria.
	
	For Alexander duals, Herzog, Hibi, and Zheng described relation trees for complementary facet ideals of quasi-forests \cite{HHZ04}, while Faridi and Hersey studied tree-supported resolutions of monomial ideals of projective dimension one \cite{FaridiHersey}. We use this framework in a fixed leaf order so that the multidegrees, total shifts, and matrix entries can be written directly from the facet and separator data.
	
	Let $\Delta_{\mathbf r}=\Delta_{\mathbf r}(n_1,\ldots,n_e)$ be the clique complex of a chordal graph with at least two maximal cliques. Fix a leaf order $S_1,\ldots,S_e$ of its facets and put $n_m=|S_m|$. For $m\ge2$, choose a branch $S_{p(m)}$ of $S_m$ among the preceding facets and set
	\[
	Q_{m-1}=S_m\cap(S_1\cup\cdots\cup S_{m-1})=S_m\cap S_{p(m)},
	\qquad r_{m-1}=|Q_{m-1}|.
	\]
	Thus $N_{\mathbf r}=\sum_{m=1}^e n_m-\sum_{m=1}^{e-1}r_m$ is the number of vertices. General formulas describe how Betti numbers change when an arbitrary simplicial complex is replaced by a skeleton \cite{Namiq2}, while Fr\"oberg's thick tree calculation gives explicit formulas when every attachment is along one vertex \cite{Fr2}. Theorem~\ref{thm:1} gives an explicit form of this relation for chordal clique complexes, where the separators $Q_m$ may have arbitrary cardinalities. Its two nonzero rows have different origins: the linear row is inherited from the $2$-linear resolution of $I_{\Delta_{\mathbf r}}$, whereas the row with $j-i=k+1$ arises from passing to the $k$-skeleton. Thus the homological effect of the skeleton operation is expressed directly in terms of the facet and separator data.
	
	Once these two rows are known, the homological invariants become structural rather than merely numerical. Theorem~\ref{thm:2} shows that the minimum separator size $r_{\min}$ is the threshold determining projective dimension and depth, and Corollaries~\ref{cor:3} and \ref{cor:4} identify exactly where Cohen--Macaulayness and initially Cohen--Macaulayness occur. Theorem~\ref{thm:3} gives a topological interpretation of the same Betti table: every proper skeleton of positive dimension has integral homology concentrated in degrees $0$ and $k$, its top homology is torsion free, and the two extremal Betti numbers correspond to two different parts of the facet and separator data. In the Cohen--Macaulay range this also gives the type and, together with Hibi's general level theorem \cite{HerzogHibi99,HibiLevel}, the Gorenstein consequences. Thus the skeleton formulas connect the separator structure not only with Betti numbers but also with depth, type, and topology. The multiplicity and independence polynomial in Corollaries~\ref{cor:1} and \ref{cor:2} are further consequences of the same face count.
	
	The Alexander dual requires a different argument. Although relation tree resolutions for complementary facet ideals of quasi-forests are known \cite{FaridiHersey,HHZ04}, the Hilbert numerator may contain cancellations between generator and first syzygy contributions when they occur in the same internal degree. For the later calculations it is therefore necessary to determine the actual multidegrees rather than only the Hilbert series. Theorem~\ref{thm:4} gives the minimal resolution in the chosen leaf order, with every shift and differential entry expressed through $S_m$ and $Q_m$. This explicit form shows that the multigraded Betti numbers in homological degree two determine the separators with multiplicity, so the separator multiset is an invariant of the multigraded Stanley--Reisner ring; the same resolution also gives the graded Betti numbers, canonical module, and the level and Gorenstein criteria in Corollaries~\ref{cor:7}--\ref{cor:10}. Corollary~\ref{cor:11} then identifies the homotopy type and integral homology of the Alexander dual itself.
	
	Proper skeletons of the dual provide a second place where the general theory and the chordal structure interact. Their Cohen--Macaulayness and levelness follow from general results on skeletons of Cohen--Macaulay complexes and are therefore used here as known results rather than presented as new conclusions. Theorem~\ref{thm:5} instead gives their complete graded Betti tables in the facet and separator parameters, and Corollaries~\ref{cor:13} and~\ref{cor:14} give the exact type, integral topology, a lower bound for the type and a formula for the last proper skeleton, and the graph-theoretic Gorenstein classification. Theorem~\ref{thm:6} identifies precisely when the dual skeleton coincides with the corresponding skeleton of the full simplex; Corollary~\ref{cor:16} translates this combinatorial threshold into a linear resolution criterion.
	
	Theorem~\ref{thm:7} considers the inverse problem: which facet and separator cardinalities are determined by the resolution of the dual? The graded Betti table determines the facet and separator cardinalities up to one common shift, and it determines both multisets exactly once the number of vertices is fixed. The common shift cannot be determined from the graded Betti table because coning changes $N_{\mathbf r}$, all $n_m$, and all $r_m$ by the same amount while preserving the resolution degrees $N_{\mathbf r}-n_m$ and $N_{\mathbf r}-r_m$. Thus the reconstruction statement is sharp at the level of graded Betti data, while Corollary~\ref{cor:7} shows that the multigraded table determines the finer multigraded data of the actual facets and separators.
	
	Finally, Section~\ref{sec:8} compares the facet and separator Hilbert series with the $f$-vector expansion. This produces a binomial convolution expressed in the facet and separator data. Corollary~\ref{cor:18} is the specialization at the actual vertex count, while Theorem~\ref{thm:8} gives the polynomial form. Since the underlying Chu--Vandermonde convolution is classical, these identities are included as consequences of the main formulas rather than as independent principal results.
	
	The paper is organized as follows. Section~\ref{sec:2} fixes notation and proves the general skeleton relation. Section~\ref{sec:3} develops the leaf order model and the face count. Section~\ref{sec:4} treats the skeletons of $\Delta_{\mathbf r}$. Sections~\ref{sec:5} and \ref{sec:6} study the Alexander dual and its skeletons. Section~\ref{sec:7} gives examples. Section~\ref{sec:8} gives the binomial convolution.
	
	\section{Preliminaries}\label{sec:2}
	
	Let $X=\{x_1,\ldots,x_N\}$ be a finite vertex set, and let $R=\K[x_1,\ldots,x_N]$ be the standard graded polynomial ring over a field $\K$.
	
	\subsection{Simplicial complexes, skeletons, and duality}
	
	A simplicial complex $\Sigma$ on $X$ is a collection of subsets of $X$ that is closed under inclusion. Its elements are called faces, and its inclusion-maximal faces are called facets. The dimension of a face $F$ is $|F|-1$, and $\dim\Sigma$ is the largest dimension of a face. The $k$-skeleton $\Sigma_{(k)}$ consists of all faces of dimension at most $k$. We use the convention $\Sigma_{(-1)}=\{\varnothing\}$; thus vertices outside all nonempty faces are allowed as ghost vertices when a skeleton is viewed on the fixed set $X$.
	
	The Alexander dual of $\Sigma$ on $X$ is
	\[
	\Sigma^\vee=\{F\subseteq X:X\setminus F\notin\Sigma\}.
	\]
	The clique complex of a graph consists of its cliques, and the independence complex consists of its independent vertex sets. A graph is chordal if every cycle of length at least four has a chord; it is co-chordal if its complement is chordal.
	
	The $f$-vector of $\Sigma$ is denoted by $f(\Sigma)=(f_{-1},f_0,\ldots)$, where $f_q$ is the number of $q$-dimensional faces and $f_{-1}=1$.
	
	\subsection{Stanley--Reisner rings and Betti numbers}
	
	The Stanley--Reisner ideal $I_\Sigma\subseteq R$ is generated by the squarefree monomials corresponding to the minimal nonfaces of $\Sigma$, and $\SR{\Sigma}=R/I_\Sigma$. For a graded $R$-module $M$, write $\beta_{i,j}^R(M)$ for its graded Betti numbers. We suppress the superscript $R$ when the polynomial ring is clear, but we always specify the module: in particular, $\beta_{i,j}(I_\Sigma)$ and $\beta_{i,j}(\SR{\Sigma})$ are not interchangeable. We use
	\[
	\pdim M=\max\{i:\beta_{i,j}(M)\ne0\text{ for some }j\},
	\qquad
	\reg M=\max\{j-i:\beta_{i,j}(M)\ne0\}.
	\]
	For $A\subseteq X$, put $x_A=\prod_{x\in A}x$. The squarefree multigraded Betti number in multidegree $x_A$ is denoted $\beta_{i,A}(M)$.
	
	A homogeneous ideal generated in degree $d$ has a $d$-linear resolution if $\beta_{i,j}(I)=0$ whenever $j\ne i+d$. Accordingly, if $I_\Sigma$ has a $d$-linear resolution, the Betti numbers of the quotient $\SR{\Sigma}$ in positive homological degrees lie in degrees $j=i+d-1$.
	
	Hochster's formula will be used in the form
	\[
	\beta_{i,j}(\SR{\Sigma})
	=
	\sum_{\substack{W\subseteq X\\ |W|=j}}
	\dim_{\K}\widetilde H_{j-i-1}(\Sigma_W;\K),
	\]
	where $\Sigma_W$ is the induced subcomplex on $W$ \cite{Ho77}.
	
	\subsection{Hilbert series and Cohen--Macaulay properties}
	
	If $d=\dim\SR{\Sigma}$, then
	\[
	H_{\SR{\Sigma}}(t)
	=
	\sum_{q=0}^{d}\frac{f_{q-1}t^q}{(1-t)^q}
	=
	\frac{\sum_{i,j}(-1)^i\beta_{i,j}(\SR{\Sigma})t^j}{(1-t)^N}
	=
	\frac{h_{\SR{\Sigma}}(t)}{(1-t)^d}.
	\]
	The multiplicity is $e(\SR{\Sigma})=h_{\SR{\Sigma}}(1)$; for a Stanley--Reisner ring, it is the number of facets of maximal dimension. The reduced Euler characteristic is $\widetilde\chi(\Sigma)=\sum_{q\ge-1}(-1)^qf_q$.
	
	A finitely generated graded module $M$ is Cohen--Macaulay if $\operatorname{depth}M=\dim M$. Following \cite{Namiq1}, its initial dimension is
	\[
	\indim M=\min\{\dim R/P:P\in\operatorname{Ass}_R(M)\},
	\]
	and $M$ is initially Cohen--Macaulay if $\operatorname{depth}M=\indim M$. For a Stanley--Reisner ring, the associated primes correspond to the facets, so
	\[
	\indim\SR{\Sigma}=\min\{|F|:F\text{ is a facet of }\Sigma\}.
	\]
	We will not use results from \cite{Namiq1} to verify the calculations below; the initially Cohen--Macaulay criteria will follow directly from this formula and the independently computed depth.
	
	By the Eagon--Reiner theorem, $I_\Sigma$ has a linear resolution if and only if $\SR{\Sigma^\vee}$ is Cohen--Macaulay \cite{ER98}. We also use the Auslander--Buchsbaum identity \cite{AB57}:
	\[
	\operatorname{depth}\SR{\Sigma}=N-\pdim\SR{\Sigma}.
	\]
	
	Throughout the paper,
	\[
	\binom ab=0\qquad\text{if }a<0,\ b<0,\text{ or }b>a,
	\]
	and otherwise $\binom ab$ has its usual combinatorial meaning.
	
	We now isolate the general relation between the Betti numbers of a complex and those of a proper skeleton.
	
	\begin{proposition}\label{prop:1}
		Let $\Sigma$ be a simplicial complex on $N$ vertices, and let $0\le k<\dim\Sigma$. Then the following statements hold.
		\begin{enumerate}
			\item[(i)] If $j<k+1$, then
			\[
			\beta_{i,i+j}(\SR{\Sigma_{(k)}})=\beta_{i,i+j}(\SR{\Sigma}).
			\]
			\item[(ii)] If $j>k+1$, then $\beta_{i,i+j}(\SR{\Sigma_{(k)}})=0$.
			\item[(iii)] For $i\ge1$ and $s=i+k+1$,
			\[
			\beta_{i,s}(\SR{\Sigma_{(k)}})
			=(-1)^i\sum_{q=0}^{k+1}(-1)^{s-q}
			\binom{N-q}{s-q}f_{q-1}(\Sigma)
			-(-1)^i\sum_{i'>i}(-1)^{i'}\beta_{i',s}(\SR{\Sigma}).
			\]
		\end{enumerate}
	\end{proposition}
	
	\begin{proof}
		Fix $W\subseteq X$. Restriction and passage to a skeleton commute, so $(\Sigma_{(k)})_W=(\Sigma_W)_{(k)}$. Passing from $\Sigma_W$ to its $k$-skeleton removes only simplices of dimension greater than $k$. The reduced homology is therefore unchanged in degrees below $k$, while the $k$-skeleton has no reduced homology above degree $k$. Hochster's formula gives parts (i) and (ii).
		
		It remains to determine the row with $s-i=k+1$. Put $s=i+k+1$. From the standard $f$-vector expression for the Hilbert series, the coefficient of $t^s$ in the Hilbert numerator of $\SR{\Sigma_{(k)}}$ is
		\[
		c_s=\sum_{q=0}^{k+1}(-1)^{s-q}\binom{N-q}{s-q}f_{q-1}(\Sigma).
		\]
		The expression of the same coefficient in terms of graded Betti numbers is
		\[
		c_s=\sum_{i'\ge0}(-1)^{i'}\beta_{i',s}(\SR{\Sigma_{(k)}}).
		\]
		If $i'<i$, then $s-i'>k+1$, and part (ii) implies that this term is zero. If $i'>i$, then $s-i'<k+1$, and part (i) identifies the term with $\beta_{i',s}(\SR{\Sigma})$. Thus only the term with $i'=i$ remains unknown. Solving for it gives part (iii).
	\end{proof}
	
	\section{Chordal clique complexes in a leaf order}\label{sec:3}
	
	Let $\Delta_{\mathbf r}(n_1,\ldots,n_e)$ be the clique complex of a chordal graph with at least two maximal cliques. Such a complex is a quasi-forest \cite{HHZ04,HHMTZ08}. Fix a leaf order $S_1,\ldots,S_e$ of its facets and put $n_m=|S_m|$. For each $m=2,\ldots,e$, choose a branch $S_{p(m)}$ of $S_m$, with $p(m)<m$. The leaf property says that
	\[
	S_m\cap S_i\subseteq S_m\cap S_{p(m)}\qquad(i<m),
	\]
	or, equivalently,
	\[
	S_m\cap(S_1\cup\cdots\cup S_{m-1})=S_m\cap S_{p(m)}.
	\]
	For $m=1,\ldots,e-1$, define
	\[
	Q_m=(S_1\cup\cdots\cup S_m)\cap S_{m+1}
	=S_{p(m+1)}\cap S_{m+1},
	\qquad r_m=|Q_m|.
	\]
	Empty separators are allowed. The number of vertices is
	\[
	N_{\mathbf r}=\sum_{m=1}^e n_m-\sum_{m=1}^{e-1}r_m.
	\]
	We also set
	\[
	\begin{aligned}
		r_{\min}&=\min\{r_1,\ldots,r_{e-1}\},
		&\qquad n_{\min}=\min\{n_1,\ldots,n_e\},\\
		n_{\max}&=\max\{n_1,\ldots,n_e\},
		&\qquad d=\dim\Delta_{\mathbf r}=n_{\max}-1.
	\end{aligned}
	\]
	When the facet sizes are understood, we write simply $\Delta_{\mathbf r}$. If $e=1$, then the complex is a simplex and the statements reduce to the standard formulas for skeletons of a simplex. We therefore assume $e\ge2$ throughout the main results.
	Throughout the rest of the paper, unless otherwise stated, the notation $S_m$, $n_m$, $Q_m$, $r_m$, $N_{\mathbf r}$, $r_{\min}$, $n_{\min}$, $n_{\max}$, and $d$ has the meaning specified above.
	
	\begin{remark}\label{rem:1}
		The notation $\Delta_{\mathbf r}(n_1,\ldots,n_e)$ specifies the numerical data associated with a chosen leaf order. It does not determine the simplicial complex up to isomorphism. Different clique trees can have the same multisets of facet and separator sizes; Example~\ref{ex:1} illustrates this point.
	\end{remark}
	
	The face numbers are obtained by counting what each new facet contributes. The corresponding Hilbert series is Fr\"oberg's formula for fat forests \cite{Fr1}.
	
	\begin{proposition}\label{prop:2}
		Let $\Delta_{\mathbf r}=\Delta_{\mathbf r}(n_1,\ldots,n_e)$ be the clique complex of a chordal graph with at least two maximal cliques. Fix a leaf order $S_1,\ldots,S_e$, put $n_m=|S_m|$, and let $r_1,\ldots,r_{e-1}$ be the corresponding separator cardinalities. For every integer $t\ge0$,
		\[
		f_{t-1}(\Delta_{\mathbf r})
		=
		\sum_{m=1}^e\binom{n_m}{t}
		-
		\sum_{m=1}^{e-1}\binom{r_m}{t}.
		\]
		Moreover,
		\[
		H_{\SR{\Delta_{\mathbf r}}}(z)
		=
		\sum_{m=1}^e\frac{1}{(1-z)^{n_m}}
		-
		\sum_{m=1}^{e-1}\frac{1}{(1-z)^{r_m}}.
		\]
	\end{proposition}
	
	\begin{proof}
		Suppose that $S_{m+1}$ is attached to $S_1\cup\cdots\cup S_m$. Every face of the simplex on $S_{m+1}$ enters the count, but the faces of the common simplex $Q_m$ have already been counted. Thus the $m$th attachment changes the number of faces of cardinality $t$ by $\binom{n_{m+1}}t-\binom{r_m}t$. Starting with $S_1$ and summing these changes proves the face formula. The same inclusion--exclusion argument, applied to the Hilbert series of each simplex and each common face, gives the second formula.
	\end{proof}
	
	Two numerical consequences will be used later.
	
	\begin{corollary}\label{cor:1}
		Under the hypotheses and notation of Proposition~\ref{prop:2}, for every $-1\le k<\dim\Delta_{\mathbf r}$, the ring $\SR{(\Delta_{\mathbf r})_{(k)}}$ has Krull dimension $k+1$ and multiplicity
		\[
		e\bigl(\SR{(\Delta_{\mathbf r})_{(k)}}\bigr)
		=
		\sum_{m=1}^e\binom{n_m}{k+1}
		-
		\sum_{m=1}^{e-1}\binom{r_m}{k+1}.
		\]
	\end{corollary}
	
	\begin{proof}
		Because $k<\dim\Delta_{\mathbf r}$, the complex contains a face of cardinality $k+1$, so the skeleton has dimension $k$. Its facets of maximal dimension are exactly the $k$-dimensional faces of $\Delta_{\mathbf r}$. The multiplicity of a Stanley--Reisner ring is the number of facets of maximal dimension, each counted once \cite[Section~5.1]{BH98}. Hence the multiplicity is $f_k(\Delta_{\mathbf r})$, and Proposition~\ref{prop:2} with $t=k+1$ gives the formula.
	\end{proof}
	
	\begin{corollary}\label{cor:2}
		Under the hypotheses and notation of Proposition~\ref{prop:2}, let $G$ be the co-chordal graph whose independence complex is $\Delta_{\mathbf r}$. If $i_t$ denotes the number of independent sets of $G$ having cardinality $t$, then
		\[
		I(G,x)=\sum_{t\ge0}i_tx^t
		=
		\sum_{m=1}^e(1+x)^{n_m}
		-
		\sum_{m=1}^{e-1}(1+x)^{r_m}.
		\]
	\end{corollary}
	
	\begin{proof}
		Independent sets of cardinality $t$ in $G$ are the faces of $\Delta_{\mathbf r}$ having cardinality $t$. Thus $i_t=f_{t-1}(\Delta_{\mathbf r})$. Summing the formula in Proposition~\ref{prop:2} over $t$ and applying the binomial theorem gives the result.
	\end{proof}
	
	The complex $\Delta_{\mathbf r}$ is flag, and its $1$-skeleton is chordal. Hence $I_{\Delta_{\mathbf r}}$ has a $2$-linear resolution by Fr\"oberg's theorem \cite{Fr}. Therefore the Betti numbers of $\SR{\Delta_{\mathbf r}}$ in positive homological degrees occur only in the linear row $j=i+1$.
	
	\begin{remark}\label{rem:2}
		A thick tree is the special case in which every new simplex is attached along one vertex. In the present notation, this means $r_1=\cdots=r_{e-1}=1$, with $n_m\ge2$. Thus the skeleton formulas below specialize to the formulas in \cite{Fr2}, while the present setting permits clique separators of arbitrary cardinality, including empty separators.
	\end{remark}
	
	\section{Skeletons of the chordal clique complex}\label{sec:4}
	
	The following elementary convolution simplifies the linear row.
	
	\begin{lemma}\label{lem:1}
		Let $n,s,A$ be nonnegative integers with $0\le A\le n$. Then
		\[
		\sum_{t=0}^s(-1)^t\binom At\binom{n-t}{s-t}
		=
		\binom{n-A}{s}.
		\]
	\end{lemma}
	
	\begin{proof}
		Taking the coefficient of $y^s$, we obtain
		\begin{align*}
			\sum_{t=0}^s(-1)^t\binom At\binom{n-t}{s-t}
			&=[y^s](1+y)^n
			\sum_{t=0}^{A}\binom At\left(\frac{-y}{1+y}\right)^t\\
			&=[y^s](1+y)^n\left(1-\frac{y}{1+y}\right)^A\\
			&=[y^s](1+y)^{n-A}
			=\binom{n-A}{s}.
		\end{align*}
	\end{proof}
	
	For a proper skeleton of positive dimension, Proposition~\ref{prop:1} leaves only two possible rows: the linear row $j=i+1$ inherited from the full complex and the row $j=i+k+1$ arising from passage to the $k$-skeleton. When $k=0$, these rows coincide.
	
	\begin{theorem}\label{thm:1}
		Let $\Delta_{\mathbf r}=\Delta_{\mathbf r}(n_1,\ldots,n_e)$ be the clique complex of a chordal graph with at least two maximal cliques. Fix a leaf order $S_1,\ldots,S_e$ of its facets, put $n_m=|S_m|$, and let $r_1,\ldots,r_{e-1}$ be the corresponding separator cardinalities.
		\begin{enumerate}
			\item[(a)] For the $(-1)$-skeleton,
			\[
			\beta_{i,i}\bigl(\SR{(\Delta_{\mathbf r})_{(-1)}}\bigr)=\binom{N_{\mathbf r}}i
			\qquad(0\le i\le N_{\mathbf r}),
			\]
			and all other graded Betti numbers vanish.
			
			\item[(b)] For the $0$-skeleton,
			\[
			\beta_{i,i+1}\bigl(\SR{(\Delta_{\mathbf r})_{(0)}}\bigr)
			=i\binom{N_{\mathbf r}}{i+1}
			\qquad(1\le i\le N_{\mathbf r}-1),
			\]
			and all other Betti numbers in positive homological degrees vanish.
			
			\item[(c)] Let $1\le k<d$. For every $i\ge1$, the only possible nonzero Betti numbers in positive homological degrees are
			\[
			\beta_{i,i+1}\bigl(\SR{(\Delta_{\mathbf r})_{(k)}}\bigr)
			=
			-\sum_{m=1}^e\binom{N_{\mathbf r}-n_m}{i+1}
			+\sum_{m=1}^{e-1}\binom{N_{\mathbf r}-r_m}{i+1}
			\]
			and
			\[
			\beta_{i,i+k+1}\bigl(\SR{(\Delta_{\mathbf r})_{(k)}}\bigr)
			=
			\sum_{t=k+2}^{i+k+1}
			(-1)^{k-t}
			\binom{N_{\mathbf r}-t}{i+k+1-t}
			\left(
			\sum_{m=1}^e\binom{n_m}{t}
			-
			\sum_{m=1}^{e-1}\binom{r_m}{t}
			\right).
			\]
			
			\item[(d)] For $k=d$, the skeleton is the full complex. Its Betti numbers in positive homological degrees are given by the linear row formula in part~(c), and all other Betti numbers in positive homological degrees vanish.
		\end{enumerate}
	\end{theorem}
	
	\begin{proof}
		Part~(a) follows from $\SR{(\Delta_{\mathbf r})_{(-1)}}\cong\K$: as an $R$-module, $\K$ has the Koszul resolution.
		
		For part~(b), take $W\subseteq X$ with $|W|=i+1$. The induced subcomplex $((\Delta_{\mathbf r})_{(0)})_W$ consists of $i+1$ isolated vertices, so its reduced zeroth homology has dimension $i$. Hochster's formula gives $\beta_{i,i+1}=i\binom{N_{\mathbf r}}{i+1}$.
		
		Now assume $1\le k<d$. Proposition~\ref{prop:1} says that every row $\beta_{i,i+j}$ with $j<k+1$ agrees with the corresponding row of the full complex and that every row with $j>k+1$ vanishes. Since $I_{\Delta_{\mathbf r}}$ has a $2$-linear resolution, $\SR{\Delta_{\mathbf r}}$ has only the linear row $j=i+1$. Hence the $k$-skeleton can have only the linear row and the row $j=i+k+1$.
		
		To compute the linear row, take the coefficient of $z^{i+1}$ in the Hilbert numerator of the full complex. Proposition~\ref{prop:2} gives
		\[
		\beta_{i,i+1}(\SR{\Delta_{\mathbf r}})
		=
		\sum_{t=0}^{i+1}(-1)^{1-t}
		\binom{N_{\mathbf r}-t}{i+1-t}
		\left(
		\sum_{m=1}^e\binom{n_m}{t}
		-
		\sum_{m=1}^{e-1}\binom{r_m}{t}
		\right).
		\]
		Apply Lemma~\ref{lem:1}, first with $A=n_m$ and then with $A=r_m$. This yields the linear row formula in part~(c), and Proposition~\ref{prop:1}(i) transfers it to the $k$-skeleton.
		
		For the row $j=i+k+1$, compare the coefficient of $z^{i+k+1}$ in the Hilbert numerators of the $k$-skeleton and the full complex. Their face contributions through cardinality $k+1$ cancel. The remaining terms from the full complex have cardinality at least $k+2$, and Proposition~\ref{prop:1}(iii) gives
		\[
		\beta_{i,i+k+1}\bigl(\SR{(\Delta_{\mathbf r})_{(k)}}\bigr)
		=
		\sum_{t=k+2}^{i+k+1}(-1)^{k-t}
		\binom{N_{\mathbf r}-t}{i+k+1-t}f_{t-1}(\Delta_{\mathbf r}).
		\]
		Substituting Proposition~\ref{prop:2} gives the stated formula. Finally, when $k=d$, the skeleton is $\Delta_{\mathbf r}$ itself, and only the linear row remains.
	\end{proof}
	
	\begin{remark}\label{rem:3}
		The separate statement for $k=0$ is essential. In part~(c), the linear row has degree shift $1$ and the additional row has degree shift $k+1$; they are distinct only when $k\ge1$. For example, if two triangles are glued along an edge, then $N_{\mathbf r}=4$, and the zero-dimensional skeleton satisfies $\beta_{1,2}=\binom42=6$, as required by Theorem~\ref{thm:1}(b).
	\end{remark}
	
	The preceding formulas determine the standard homological invariants.
	
	\begin{theorem}\label{thm:2}
		Let $\Delta_{\mathbf r}=\Delta_{\mathbf r}(n_1,\ldots,n_e)$ be the clique complex of a chordal graph with at least two maximal cliques $S_1,\ldots,S_e$ in a fixed leaf order. Put $n_m=|S_m|$, let $r_1,\ldots,r_{e-1}$ be the corresponding separator cardinalities.
		Let $-1\le k\le d$. Then
		\[
		\reg\SR{(\Delta_{\mathbf r})_{(k)}}
		=
		\begin{cases}
			k+1,&-1\le k<d,\\
			1,&k=d,
		\end{cases}
		\]
		and
		\[
		\pdim\SR{(\Delta_{\mathbf r})_{(k)}}
		=
		\begin{cases}
			N_{\mathbf r}-k-1,&-1\le k\le r_{\min},\\
			N_{\mathbf r}-r_{\min}-1,&r_{\min}<k\le d.
		\end{cases}
		\]
		Consequently,
		\[
		\operatorname{depth}\SR{(\Delta_{\mathbf r})_{(k)}}
		=
		\begin{cases}
			k+1,&-1\le k\le r_{\min},\\
			r_{\min}+1,&r_{\min}<k\le d.
		\end{cases}
		\]
	\end{theorem}
	
	\begin{proof}
		For $k=-1$, Theorem~\ref{thm:1}(a) gives regularity $0$. Let $0\le k<d$. Proposition~\ref{prop:1} implies that every row of degree shift greater than $k+1$ vanishes. Choose a face $F$ of $\Delta_{\mathbf r}$ with $|F|=k+2$. The induced subcomplex of the $k$-skeleton on $F$ is the boundary of a $(k+1)$-simplex, so Hochster's formula gives $\beta_{1,k+2}\ne0$. Thus the regularity is exactly $k+1$. For $k=d$, the ideal $I_{\Delta_{\mathbf r}}$ has a $2$-linear resolution, and the quotient has regularity $1$.
		
		Fr\"oberg's depth formula for fat forests gives $\operatorname{depth}\SR{\Delta_{\mathbf r}}=r_{\min}+1$ \cite{Fr1}. For every simplicial complex $\Sigma$,
		\[
		\operatorname{depth}\SR{\Sigma}
		=1+\max\{q:\SR{\Sigma_{(q)}}\text{ is Cohen--Macaulay}\};
		\]
		see \cite[Exercise~5.1.23]{BH98} and \cite{RV15}. It follows that $\SR{(\Delta_{\mathbf r})_{(k)}}$ is Cohen--Macaulay for $k\le r_{\min}$. Its dimension is $k+1$, so Auslander--Buchsbaum gives $\pdim=N_{\mathbf r}-k-1$ in this range.
		
		Suppose now that $k>r_{\min}$. The row $j=i+k+1$ in Theorem~\ref{thm:1} can occur only when $i+k+1\le N_{\mathbf r}$, and hence only in homological degrees $i\le N_{\mathbf r}-k-1<N_{\mathbf r}-r_{\min}-1$. We show that the linear row is nonzero at $i=N_{\mathbf r}-r_{\min}-1$.
		
		Every facet is incident with an edge of the relation tree because $e\ge2$. The separator on such an edge is a proper subset of that facet. Hence $r_{\min}\le n_m-1$ for every $m$. At $i=N_{\mathbf r}-r_{\min}-1$, all terms $\binom{N_{\mathbf r}-n_m}{i+1}$ vanish, and a separator contributes $1$ precisely when its cardinality is $r_{\min}$. The corresponding Betti number in the linear row is therefore the positive number of separators of minimum size. Thus the projective dimension is $N_{\mathbf r}-r_{\min}-1$. Auslander--Buchsbaum gives the depth formula.
	\end{proof}
	
	We use the standard terminology that a Betti number $\beta_{i,i+j}(M)$ is \emph{extremal} if
	\[
	\beta_{p,p+q}(M)=0
	\quad\text{for all }p\ge i,\ q\ge j,
	\quad (p,q)\ne(i,j).
	\]
	The two-row form in Theorem~\ref{thm:1} identifies these extremal positions explicitly.
	
	\begin{theorem}\label{thm:3}
		Let $\Delta_{\mathbf r}=\Delta_{\mathbf r}(n_1,\ldots,n_e)$ be the clique complex of a chordal graph with at least two maximal cliques $S_1,\ldots,S_e$ in a fixed leaf order. Put $n_m=|S_m|$, let $r_1,\ldots,r_{e-1}$ be the corresponding separator cardinalities. Let $c$ be the number of connected components of $\Delta_{\mathbf r}$. Every connected component of $\Delta_{\mathbf r}$ is contractible. For $1\le k<d$, put
		\[
		\lambda_k=
		\sum_{m=1}^{e}\binom{n_m-1}{k+1}
		-
		\sum_{m=1}^{e-1}\binom{r_m-1}{k+1},
		\qquad
		\mu=\bigl|\{m:r_m=r_{\min}\}\bigr|.
		\]
		Then the following statements hold.
		\begin{enumerate}
			\item[(i)] Each connected component of $(\Delta_{\mathbf r})_{(k)}$ is homotopy equivalent to a wedge of $k$-spheres, and the total number of spheres over all components is $\lambda_k$. Consequently,
			\[
			\widetilde H_q((\Delta_{\mathbf r})_{(k)};\mathbb Z)
			\cong
			\begin{cases}
				\mathbb Z^{c-1},&q=0,\\
				\mathbb Z^{\lambda_k},&q=k,\\
				0,&\text{otherwise}.
			\end{cases}
			\]
			In particular, the integral homology of every proper skeleton of positive dimension is torsion free.
			
			\item[(ii)] The last nonzero Betti numbers in the two rows are
			\[
			\beta_{N_{\mathbf r}-k-1,N_{\mathbf r}}
			\bigl(\SR{(\Delta_{\mathbf r})_{(k)}}\bigr)=\lambda_k,\quad \beta_{N_{\mathbf r}-r_{\min}-1,\,N_{\mathbf r}-r_{\min}}
			\bigl(\SR{(\Delta_{\mathbf r})_{(k)}}\bigr)=\mu.
			\]
			If $k\le r_{\min}$, the first of these is the unique extremal Betti number. If $k>r_{\min}$, these two Betti numbers are exactly the two extremal Betti numbers.
			
			\item[(iii)] If $k\le r_{\min}$, then $\SR{(\Delta_{\mathbf r})_{(k)}}$ is Cohen--Macaulay and
			\[
			\operatorname{type}\SR{(\Delta_{\mathbf r})_{(k)}}
			=
			\begin{cases}
				\lambda_k,&k<r_{\min},\\
				\lambda_k+\mu,&k=r_{\min}.
			\end{cases}
			\]
			For $1\le k<d$, the ring $\SR{(\Delta_{\mathbf r})_{(k)}}$ is level if and only if $k<r_{\min}$. If $d\ge1$, the $0$-skeleton is level of type $N_{\mathbf r}-1$. No proper skeleton of $\Delta_{\mathbf r}$ with nonnegative dimension is Gorenstein.
		\end{enumerate}
	\end{theorem}
	
	\begin{proof}
		Add the facets in leaf order. If $Q_m\ne\varnothing$, the simplex $S_{m+1}$ is attached to the preceding union along the nonempty face $Q_m$ and collapses onto that face relative to the preceding union. If $Q_m=\varnothing$, the new simplex forms a new connected component. It follows inductively that every connected component of $\Delta_{\mathbf r}$ is contractible. The relation tree has exactly $c-1$ edges with empty separator.
		
		Fix $1\le k<d$. The $k$-skeleton has the same connected components as $\Delta_{\mathbf r}$. If a component has dimension at most $k$, it is unchanged and hence contractible. Otherwise, the component is obtained from its $k$-skeleton by adjoining simplices of dimension at least $k+1$. Therefore the inclusion of the $k$-skeleton induces isomorphisms on homotopy groups below degree $k$. For $k\ge2$, each component of the skeleton is consequently $(k-1)$-connected; for $k=1$, each component is a connected graph. A connected graph is a wedge of circles. For $k\ge2$, the Hurewicz theorem identifies $\pi_k$ with the free abelian group $H_k$, and maps from $k$-spheres representing a basis of $H_k$ give a homology equivalence from a wedge of $k$-spheres to the component. Since both spaces are simply connected CW complexes, Whitehead's theorem makes this a homotopy equivalence. Thus, in every case, a $k$-dimensional component is homotopy equivalent to a wedge of $k$-spheres.
		
		It remains to count the spheres. Proposition~\ref{prop:2}, together with
		\[
		-1+\sum_{t=1}^{k+1}(-1)^{t-1}\binom At
		=(-1)^k\binom{A-1}{k+1}
		\qquad(A\ge1),
		\]
		gives
		\[
		\widetilde\chi((\Delta_{\mathbf r})_{(k)})
		=(c-1)+(-1)^k\lambda_k.
		\]
		Here each empty separator contributes to the term $c-1$, while under our binomial convention its contribution to $\lambda_k$ is zero. Since the only possible reduced homology groups are in degrees $0$ and $k$, the total rank in degree $k$ is $\lambda_k$. This proves part~(i).
		
		Hochster's formula on the full vertex set gives
		\[
		\beta_{N_{\mathbf r}-k-1,N_{\mathbf r}}
		\bigl(\SR{(\Delta_{\mathbf r})_{(k)}}\bigr)
		=\dim_{\K}\widetilde H_k((\Delta_{\mathbf r})_{(k)};\K)
		=\lambda_k.
		\]
		For the linear row, substitute $i=N_{\mathbf r}-r_{\min}-1$ into Theorem~\ref{thm:1}(c). Every facet term vanishes because $n_m\ge r_{\min}+1$, and a separator contributes one exactly when $r_m=r_{\min}$. Thus the second endpoint is $\mu$.
		
		Theorem~\ref{thm:1} has only the linear row and the row $j=i+k+1$. The latter ends at homological degree $N_{\mathbf r}-k-1$, while the former ends at $N_{\mathbf r}-r_{\min}-1$. Hence, when $k\le r_{\min}$, the last nonzero entry in the row $j=i+k+1$ dominates every other nonzero Betti number and is the unique extremal one. When $k>r_{\min}$, the two endpoints are incomparable and are precisely the two extremal Betti numbers. This proves part~(ii).
		
		Assume now that $k\le r_{\min}$. By Theorem~\ref{thm:2}, the projective dimension is $N_{\mathbf r}-k-1$. If $k<r_{\min}$, the linear row ends earlier, so the last free module is
		\[
		R(-N_{\mathbf r})^{\lambda_k}.
		\]
		Hence the ring is level of type $\lambda_k$. If $k=r_{\min}$, the two rows end in the same homological degree and the last free module contains
		\[
		R(-(N_{\mathbf r}-k))^{\mu}\oplus R(-N_{\mathbf r})^{\lambda_k}.
		\]
		Both summands are nonzero: $\mu>0$, and $\lambda_k>0$ because $k<d$, so there is a $(k+1)$-face whose boundary is a nonzero $k$-cycle in the $k$-skeleton. Thus the type is $\lambda_k+\mu$ and the ring is not level. When $k>r_{\min}$, Corollary~\ref{cor:3} below shows that the ring is not Cohen--Macaulay, so it is not level in the standard sense. Theorem~\ref{thm:1}(b) shows that the $0$-skeleton is level of type $N_{\mathbf r}-1$ whenever it is proper.
		
		Finally, suppose $k<r_{\min}$. Then
		\[
		\lambda_k=
		\binom{n_1-1}{k+1}
		+\sum_{m=2}^{e}
		\left[
		\binom{n_m-1}{k+1}
		-
		\binom{r_{m-1}-1}{k+1}
		\right]\ge e.
		\]
		Indeed, $S_1$ is incident with an edge of the relation tree, so $n_1\ge r_{\min}+1\ge k+2$, and for $m\ge2$ we have $n_m\ge r_{m-1}+1$ with $r_{m-1}>k$; Pascal's identity shows that every bracket is at least one. Hence the type is at least two. If $k=r_{\min}$, then $\lambda_k+\mu\ge2$. A proper $0$-skeleton has $N_{\mathbf r}\ge3$ and type $N_{\mathbf r}-1\ge2$. Therefore no proper skeleton with nonnegative dimension is Gorenstein.
	\end{proof}
	
	\begin{corollary}\label{cor:3}
		Under the hypotheses and notation of Theorem~\ref{thm:2}, for every $-1\le k\le d$, the following statements hold.
		\begin{enumerate}
			\item[(i)] The ring $\SR{(\Delta_{\mathbf r})_{(k)}}$ is initially Cohen--Macaulay if and only if
			\[
			k\le r_{\min}
			\qquad\text{or}\qquad
			r_{\min}=n_{\min}-1.
			\]
			Thus, if $r_{\min}<n_{\min}-1$, the initially Cohen--Macaulay skeletons are exactly those with $k\le r_{\min}$; if $r_{\min}=n_{\min}-1$, every skeleton is initially Cohen--Macaulay.
			
			\item[(ii)] The ring $\SR{(\Delta_{\mathbf r})_{(k)}}$ is Cohen--Macaulay if and only if $k\le r_{\min}$.
		\end{enumerate}
	\end{corollary}
	
	\begin{proof}
		The smallest cardinality of a facet of the $k$-skeleton is
		\[
		\indim\SR{(\Delta_{\mathbf r})_{(k)}}=\min\{k+1,n_{\min}\}.
		\]
		Indeed, if $k+1<n_{\min}$, every face of cardinality at most $k$ extends inside an original facet to a face of cardinality $k+1$, so every facet of the skeleton has that cardinality. If $k+1\ge n_{\min}$, an original facet of cardinality $n_{\min}$ remains a facet of the skeleton, and no smaller facet can occur.
		
		We also have $r_{\min}+1\le n_{\min}$. To see this, choose a facet of minimum cardinality and an incident edge in the relation tree. Its separator is a proper subset of the facet, and $r_{\min}$ is no larger than that separator.
		
		If $k\le r_{\min}$, Theorem~\ref{thm:2} gives depth $k+1$, which equals the initial dimension. Now let $k>r_{\min}$. The depth is $r_{\min}+1$, while $k+1>r_{\min}+1$ and $n_{\min}\ge r_{\min}+1$. Therefore the depth equals $\min\{k+1,n_{\min}\}$ precisely when $n_{\min}=r_{\min}+1$. This proves part~(i).
		
		For part~(ii), the Krull dimension is $k+1$. Theorem~\ref{thm:2} shows that the depth equals this dimension exactly when $k\le r_{\min}$.
	\end{proof}
	
	\begin{corollary}\label{cor:4}
		Under the hypotheses and notation of Theorem~\ref{thm:2}, for the full complex $\Delta_{\mathbf r}$, the following statements hold.
		\begin{enumerate}
			\item[(i)] The ring $\SR{\Delta_{\mathbf r}}$ is initially Cohen--Macaulay if and only if $r_{\min}=n_{\min}-1$.
			\item[(ii)] The ring $\SR{\Delta_{\mathbf r}}$ is Cohen--Macaulay if and only if all facets have a common cardinality $n$ and every separator has cardinality $n-1$.
			\item[(iii)] If the equivalent conditions in part~(ii) hold, then $\SR{\Delta_{\mathbf r}}$ is level of type $e-1$. In particular, it is Gorenstein if and only if $e=2$.
		\end{enumerate}
	\end{corollary}
	
	\begin{proof}
		Part~(i) follows from Corollary~\ref{cor:3}(i) with $k=d=n_{\max}-1$, because $r_{\min}\le n_{\min}-1\le d$. For part~(ii), Corollary~\ref{cor:3}(ii) gives Cohen--Macaulayness exactly when $d\le r_{\min}$. The inequalities
		\[
		r_{\min}\le n_{\min}-1\le n_{\max}-1=d
		\]
		then imply equality throughout. Thus every facet has the same cardinality and every separator has codimension one in the adjacent facets. The converse follows from the same criterion, and agrees with Fr\"oberg's criterion for the full fat forest \cite{Fr1}.
		
		Under these conditions, $n_m=n$, $r_m=n-1$, and $N_{\mathbf r}=n+e-1$. Hence $\pdim\SR{\Delta_{\mathbf r}}=e-1$, and the last entry in the linear row of Theorem~\ref{thm:1}(d) is
		\[
		\beta_{e-1,e}(\SR{\Delta_{\mathbf r}})=e-1.
		\]
		The last free module therefore has one shift, so the ring is level of type $e-1$. A Cohen--Macaulay ring is Gorenstein exactly when its type is one, which here is equivalent to $e=2$. This proves part~(iii).
	\end{proof}
	
	\section{The Alexander dual and its minimal resolution}\label{sec:5}
	
	We now turn to $\Delta_{\mathbf r}^\vee$. Since $I_{\Delta_{\mathbf r}}$ is generated in degree two and has a linear resolution, the Eagon--Reiner theorem shows that $\SR{\Delta_{\mathbf r}^\vee}$ is Cohen--Macaulay of dimension $N_{\mathbf r}-2$.
	
	\begin{proposition}\label{prop:3}
		Let $\Delta_{\mathbf r}=\Delta_{\mathbf r}(n_1,\ldots,n_e)$ be the clique complex of a chordal graph with at least two maximal cliques $S_1,\ldots,S_e$ in a fixed leaf order. Put $n_m=|S_m|$, let $r_1,\ldots,r_{e-1}$ be the corresponding separator cardinalities.
		For $-1\le q\le N_{\mathbf r}-3$,
		\[
		f_q(\Delta_{\mathbf r}^\vee)
		=
		\binom{N_{\mathbf r}}{q+1}
		-
		\sum_{m=1}^e\binom{n_m}{N_{\mathbf r}-q-1}
		+
		\sum_{m=1}^{e-1}\binom{r_m}{N_{\mathbf r}-q-1}.
		\]
		Moreover,
		\[
		h_{\SR{\Delta_{\mathbf r}^\vee}}(z)
		=
		\frac{
			1-\sum_{m=1}^e z^{N_{\mathbf r}-n_m}
			+\sum_{m=1}^{e-1}z^{N_{\mathbf r}-r_m}
		}{(1-z)^2}.
		\]
	\end{proposition}
	
	\begin{proof}
		A set $A\subseteq X$ of cardinality $q+1$ belongs to $\Delta_{\mathbf r}^\vee$ if and only if $X\setminus A$ is not a face of $\Delta_{\mathbf r}$. Hence
		\[
		f_q(\Delta_{\mathbf r}^\vee)
		=
		\binom{N_{\mathbf r}}{q+1}
		-f_{N_{\mathbf r}-q-2}(\Delta_{\mathbf r}),
		\]
		and Proposition~\ref{prop:2} gives the first formula.
		
		To obtain the Hilbert numerator, sum first over all subsets of $X$ and then remove the complements of the faces of $\Delta_{\mathbf r}$. After multiplication by $(1-z)^{N_{\mathbf r}}$, this gives
		\[
		P(z)
		=1-\sum_{F\in\Delta_{\mathbf r}}z^{N_{\mathbf r}-|F|}(1-z)^{|F|}.
		\]
		Apply the same leaf order inclusion--exclusion used in Proposition~\ref{prop:2}. For a simplex on a set $S$ of cardinality $a$,
		\[
		\sum_{F\subseteq S}z^{N_{\mathbf r}-|F|}(1-z)^{|F|}
		=z^{N_{\mathbf r}-a}.
		\]
		The facet terms and separator terms therefore give
		\[
		P(z)=1-\sum_{m=1}^e z^{N_{\mathbf r}-n_m}
		+\sum_{m=1}^{e-1}z^{N_{\mathbf r}-r_m}.
		\]
		Because the ring is Cohen--Macaulay of dimension $N_{\mathbf r}-2$, its Hilbert series has denominator $(1-z)^{N_{\mathbf r}-2}$. Thus $P(z)=(1-z)^2h_{\SR{\Delta_{\mathbf r}^\vee}}(z)$, which proves the second formula.
	\end{proof}
	
	\begin{corollary}\label{cor:5}
		Under the hypotheses and notation of Proposition~\ref{prop:3}, let $G$ be the chordal graph with $\Delta_{\mathbf r}=\operatorname{Cl}(G)$. The multiplicity of $\SR{\Delta_{\mathbf r}^\vee}$ is the number of nonedges of $G$, and hence
		\begin{align*}
			e\bigl(\SR{\Delta_{\mathbf r}^\vee}\bigr)
			&=|E(\overline G)|\\
			&=\binom{N_{\mathbf r}}2
			-\sum_{m=1}^{e}\binom{n_m}{2}
			+\sum_{m=1}^{e-1}\binom{r_m}{2}\\
			&=\frac12\left[
			\sum_{m=1}^{e-1}(N_{\mathbf r}-r_m)(N_{\mathbf r}-r_m-1)
			-
			\sum_{m=1}^{e}(N_{\mathbf r}-n_m)(N_{\mathbf r}-n_m-1)
			\right].
		\end{align*}
	\end{corollary}
	
	\begin{proof}
		Because $\Delta_{\mathbf r}$ is a clique complex, its minimal nonfaces are precisely the nonedges of $G$. The facets of the Alexander dual are their complements. Since $\SR{\Delta_{\mathbf r}^\vee}$ is Cohen--Macaulay of dimension $N_{\mathbf r}-2$, all facets of $\Delta_{\mathbf r}^\vee$ have cardinality $N_{\mathbf r}-2$, so the multiplicity is the number of these facets, namely $|E(\overline G)|$. Proposition~\ref{prop:2} with $t=2$ gives the second expression. For the last expression, let $P(z)$ be the Hilbert numerator from Proposition~\ref{prop:3}. Since $P(z)=(1-z)^2h(z)$, one has $P''(1)=2h(1)$, and $h(1)$ is the multiplicity. Differentiating $P(z)$ twice gives the stated equality.
	\end{proof}
	
	\begin{corollary}\label{cor:6}
		Under the hypotheses and notation of Proposition~\ref{prop:3}, write $h_{\SR{\Delta_{\mathbf r}^\vee}}(z)=\sum_{q=0}^{N_{\mathbf r}-2}h_qz^q$. Then
		\[
		h_q
		=
		(q+1)
		-
		\sum_{m=1}^e\max\{0,q-N_{\mathbf r}+n_m+1\}
		+
		\sum_{m=1}^{e-1}\max\{0,q-N_{\mathbf r}+r_m+1\}
		\]
		for $0\le q\le N_{\mathbf r}-2$.
	\end{corollary}
	
	\begin{proof}
		Expand $(1-z)^{-2}=\sum_{q\ge0}(q+1)z^q$ in Proposition~\ref{prop:3}. A term $z^a(1-z)^{-2}$ contributes $q-a+1$ to the coefficient of $z^q$ when $q\ge a$ and contributes zero otherwise. Substituting $a=N_{\mathbf r}-n_m$ and $a=N_{\mathbf r}-r_m$ gives the formula.
	\end{proof}
	
	The Hilbert numerator in Proposition~\ref{prop:3} may contain a positive and a negative term in the same degree. It therefore does not, by itself, determine how many summands occur in homological degrees one and two. To determine these contributions separately, we construct the minimal resolution.
	
	\begin{theorem}\label{thm:4}
		Let $\Delta_{\mathbf r}=\Delta_{\mathbf r}(n_1,\ldots,n_e)$ be the clique complex of a chordal graph on vertex set $X$, with at least two maximal cliques $S_1,\ldots,S_e$ in a fixed leaf order. Put $n_m=|S_m|$, let $S_{p(m)}$ be the chosen branch of $S_m$ for $m\ge2$, set $Q_{m-1}=S_m\cap S_{p(m)}$ and $r_{m-1}=|Q_{m-1}|$, and let $N_{\mathbf r}=|X|$. For $m=1,\ldots,e$, set
		\[
		g_m=\prod_{x\in X\setminus S_m}x.
		\]
		These are the minimal generators of $I_{\Delta_{\mathbf r}^\vee}$. Join each $S_m$ with $m\ge2$ to its chosen branch $S_{p(m)}$; these edges form a relation tree. Then $I_{\Delta_{\mathbf r}^\vee}$ has the minimal free resolution
		\[
		0\longrightarrow
		\bigoplus_{m=2}^e R\bigl(-(N_{\mathbf r}-r_{m-1})\bigr)
		\xrightarrow{\partial_2}
		\bigoplus_{m=1}^e R\bigl(-(N_{\mathbf r}-n_m)\bigr)
		\xrightarrow{\partial_1}
		I_{\Delta_{\mathbf r}^\vee}
		\longrightarrow0.
		\]
		Here $\partial_1(e_m)=g_m$. The column of $\partial_2$ indexed by $m\ge2$ has exactly two nonzero entries:
		\[
		(\partial_2)_{p(m),m-1}
		=
		\prod_{x\in S_{p(m)}\setminus S_m}x,
		\qquad
		(\partial_2)_{m,m-1}
		=
		-\prod_{x\in S_m\setminus S_{p(m)}}x.
		\]
	\end{theorem}
	
	\begin{proof}
		We first verify that the chosen branches form a tree. Each edge joins a later facet $S_m$ to an earlier facet $S_{p(m)}$. Repeatedly passing to the branch therefore reaches $S_1$, so the graph is connected. A cycle cannot occur because every edge decreases the facet index in one direction. Hence the graph is a tree.
		
		Now put $J_m=(g_1,\ldots,g_m)$. For $m\ge2$, the usual formula for a colon of monomial ideals gives
		\begin{align*}
			J_{m-1}:g_m
			&=
			\left(
			\frac{g_i}{\gcd(g_i,g_m)}:1\le i<m
			\right)\\
			&=
			\left(
			\prod_{x\in S_m\setminus S_i}x:1\le i<m
			\right).
		\end{align*}
		Because $S_m$ is a leaf with branch $S_{p(m)}$, we have $S_m\cap S_i\subseteq S_m\cap S_{p(m)}$ for every $i<m$. Equivalently,
		\[
		S_m\setminus S_{p(m)}\subseteq S_m\setminus S_i.
		\]
		Thus the monomial
		\[
		u_m=\prod_{x\in S_m\setminus S_{p(m)}}x
		\]
		divides every listed generator of the colon ideal. It is itself the generator corresponding to $i=p(m)$, and therefore $J_{m-1}:g_m=(u_m)$.
		
		Starting with $J_1=(g_1)$, adjoin $g_2,\ldots,g_e$ in leaf order. Since each colon ideal is principal, the mapping cone at the $m$th step adds one first syzygy and no higher syzygy. The new first syzygy is the relation between $g_{p(m)}$ and $g_m$, in multidegree
		\[
		\operatorname{lcm}(g_{p(m)},g_m)
		=
		\prod_{x\in X\setminus(S_{p(m)}\cap S_m)}x.
		\]
		Because $S_{p(m)}\cap S_m=Q_{m-1}$, its total degree is $N_{\mathbf r}-r_{m-1}$, and the two coefficients are the entries given in the theorem. The two adjacent facets are distinct maximal faces, so both set differences are nonempty. Every matrix entry therefore lies in the homogeneous maximal ideal, and the mapping cone is minimal. Induction on $m$ completes the proof.
	\end{proof}
	
	\begin{remark}\label{rem:4}
		The existence of a tree-supported resolution is not new. Herzog, Hibi, and Zheng used relation trees and the Hilbert--Burch theorem for complementary facet ideals of quasi-forests \cite[Section~2]{HHZ04}; Faridi and Hersey later characterized monomial ideals of projective dimension one by resolutions supported on trees \cite{FaridiHersey}. The role of Theorem~\ref{thm:4} is to express this resolution in the chosen leaf order, with all multidegrees, shifts, and matrix entries written directly in terms of $S_m$ and $Q_m$.
	\end{remark}
	
	\begin{corollary}\label{cor:7}
		Under the hypotheses and notation of Theorem~\ref{thm:4}, for every subset $A\subseteq X$, the nonzero multigraded Betti numbers in positive homological degrees of the quotient $\SR{\Delta_{\mathbf r}^\vee}$ are
		\[
		\beta_{1,A}(\SR{\Delta_{\mathbf r}^\vee})
		=
		\bigl|\{m:X\setminus S_m=A\}\bigr|,
		\qquad
		\beta_{2,A}(\SR{\Delta_{\mathbf r}^\vee})
		=
		\bigl|\{q:X\setminus Q_q=A\}\bigr|.
		\]
		All multigraded Betti numbers in homological degree at least three vanish. In particular, the multigraded Betti numbers in homological degree one determine the labeled complex $\Delta_{\mathbf r}$, and those in homological degree two determine the separators with multiplicity.
	\end{corollary}
	
	\begin{proof}
		Append $R\to\SR{\Delta_{\mathbf r}^\vee}\to0$ to the resolution in Theorem~\ref{thm:4}. The free summand indexed by $S_m$ then occurs in homological degree one and has squarefree multidegree $X\setminus S_m$. The summand indexed by $Q_q$ occurs in homological degree two and has squarefree multidegree $X\setminus Q_q$. Minimality prevents cancellation, so the multiplicities of these summands are the stated multigraded Betti numbers.
	\end{proof}
	
	\begin{corollary}\label{cor:8}
		Under the hypotheses and notation of Theorem~\ref{thm:4}, although the separators $Q_1,\ldots,Q_{e-1}$ are defined from a leaf order, their multiset as subsets of $X$ is independent of that order. Consequently, the multiset of separator cardinalities $\{r_1,\ldots,r_{e-1}\}$ is also independent of the leaf order.
	\end{corollary}
	
	\begin{proof}
		The multigraded Betti numbers of $\SR{\Delta_{\mathbf r}^\vee}$ are invariants of the labeled Stanley--Reisner ring. By Corollary~\ref{cor:7}, the multiset of second syzygy multidegrees is $\{X\setminus Q_q:1\le q\le e-1\}$, with multiplicity. Taking complements determines the separator multiset, and taking cardinalities gives the second assertion. For connected chordal graphs, this agrees with the classical invariance of the separator multiset of a clique tree \cite[Theorem~4.2]{BlairPeyton}.
	\end{proof}
	
	\begin{corollary}\label{cor:9}
		Under the hypotheses and notation of Theorem~\ref{thm:4}, the nonzero graded Betti numbers of $\SR{\Delta_{\mathbf r}^\vee}$ are $\beta_{0,0}=1$ and
		\[
		\beta_{1,j}=\bigl|\{m:N_{\mathbf r}-n_m=j\}\bigr|,
		\qquad
		\beta_{2,j}=\bigl|\{m:N_{\mathbf r}-r_m=j\}\bigr|.
		\]
		In particular,
		\[
		\reg\SR{\Delta_{\mathbf r}^\vee}=N_{\mathbf r}-r_{\min}-2,
		\qquad
		\operatorname{type}\SR{\Delta_{\mathbf r}^\vee}=e-1,
		\qquad
		a(\SR{\Delta_{\mathbf r}^\vee})=-r_{\min}.
		\]
	\end{corollary}
	
	\begin{proof}
		The graded Betti numbers are obtained from Corollary~\ref{cor:7} by taking total degrees. As observed in the proof of Theorem~\ref{thm:2}, $r_{\min}\le n_m-1$ for every facet. Therefore
		\[
		N_{\mathbf r}-n_m-1\le N_{\mathbf r}-r_{\min}-2,
		\]
		so the largest value of $j-i$ is attained in homological degree two at a separator of minimum cardinality. This gives the regularity. Since the quotient is Cohen--Macaulay of codimension two, its type is the total rank of the last free module, namely $e-1$. Finally, for a Cohen--Macaulay standard graded algebra, $a=\reg-\dim$, and hence
		\[
		a(\SR{\Delta_{\mathbf r}^\vee})
		=(N_{\mathbf r}-r_{\min}-2)-(N_{\mathbf r}-2)
		=-r_{\min}.
		\]
	\end{proof}
	
	\begin{corollary}\label{cor:10}
		Under the hypotheses and notation of Theorem~\ref{thm:4}, let $A=\SR{\Delta_{\mathbf r}^\vee}$. Its canonical module has a minimal presentation
		\[
		\bigoplus_{m=1}^{e}R(-n_m)
		\longrightarrow
		\bigoplus_{q=1}^{e-1}R(-r_q)
		\longrightarrow
		\omega_A
		\longrightarrow0.
		\]
		Thus the minimal generator degrees of $\omega_A$ are $r_1,\ldots,r_{e-1}$. The ring $A$ is level if and only if $r_1=\cdots=r_{e-1}$, and the following conditions are equivalent:
		\begin{enumerate}
			\item[(i)] $A$ is Gorenstein;
			\item[(ii)] $I_{\Delta_{\mathbf r}^\vee}$ is a complete intersection;
			\item[(iii)] $e=2$.
		\end{enumerate}
	\end{corollary}
	
	\begin{proof}
		Append $R\to A\to0$ to the resolution in Theorem~\ref{thm:4} and apply $\operatorname{Hom}_R(-,R(-N_{\mathbf r}))$. Since $A$ is Cohen--Macaulay of codimension two, the resulting cokernel is
		\[
		\omega_A=\operatorname{Ext}^2_R(A,R(-N_{\mathbf r})).
		\]
		Dualizing changes the shifts $N_{\mathbf r}-n_m$ and $N_{\mathbf r}-r_q$ to $n_m$ and $r_q$, respectively, and gives the stated presentation. It is minimal because the entries of $\partial_2$ lie in the homogeneous maximal ideal.
		
		The canonical module therefore has one minimal generator in degree $r_q$ for each separator. It is generated in one degree exactly when all $r_q$ are equal, which is the level criterion. The number of generators is $e-1$, so the Cohen--Macaulay ring $A$ is Gorenstein exactly when $e=2$. In that case $I_{\Delta_{\mathbf r}^\vee}$ has two minimal generators and height two, hence is a complete intersection. Conversely, a complete intersection quotient is Gorenstein.
	\end{proof}
	
	For recent structural work on level Stanley--Reisner rings of codimension two, see Pournaki, Poursoltani, Terai and Yassemi \cite{PPTZ26}. We use only the direct resolution above and do not rely on their classification.
	
	\begin{corollary}\label{cor:11}
		Under the hypotheses and notation of Theorem~\ref{thm:4}, let $c$ be the number of connected components of $\Delta_{\mathbf r}$, equivalently of its $1$-skeleton. Then $\Delta_{\mathbf r}^\vee$ is shellable, and
		\[
		\widetilde H_q(\Delta_{\mathbf r}^\vee;\mathbb Z)=0
		\qquad(q\ne N_{\mathbf r}-3),
		\]
		while
		\[
		\widetilde H_{N_{\mathbf r}-3}(\Delta_{\mathbf r}^\vee;\mathbb Z)
		\cong\mathbb Z^{c-1}.
		\]
		If $N_{\mathbf r}\ge3$, then
		\[
		\Delta_{\mathbf r}^\vee\simeq\bigvee^{c-1}S^{N_{\mathbf r}-3}.
		\]
		In particular, the dual is contractible when $\Delta_{\mathbf r}$ is connected. Its reduced Euler characteristic is
		\[
		\widetilde\chi(\Delta_{\mathbf r}^\vee)=(-1)^{N_{\mathbf r}-3}(c-1).
		\]
	\end{corollary}
	
	\begin{proof}
		The ideal $I_{\Delta_{\mathbf r}}$ is a quadratic monomial ideal with a linear resolution. Such an ideal has linear quotients by \cite{HHZPowers}. Since $I_{(\Delta_{\mathbf r}^\vee)^\vee}=I_{\Delta_{\mathbf r}}$, the equivalence between shellability and linear quotients under Alexander duality shows that $\Delta_{\mathbf r}^\vee$ is shellable \cite[Theorem~1.4(c)]{HHZ04}.
		
		By Theorem~\ref{thm:3}, every connected component of $\Delta_{\mathbf r}$ is contractible. Hence $\Delta_{\mathbf r}$ is homotopy equivalent to a disjoint union of $c$ points, and
		\[
		\widetilde H^j(\Delta_{\mathbf r};\mathbb Z)=0\quad(j\ne0),
		\qquad
		\widetilde H^0(\Delta_{\mathbf r};\mathbb Z)\cong\mathbb Z^{c-1}.
		\]
		Combinatorial Alexander duality \cite{BjornerTancer09} gives
		\[
		\widetilde H_q(\Delta_{\mathbf r}^\vee;\mathbb Z)
		\cong
		\widetilde H^{N_{\mathbf r}-q-3}(\Delta_{\mathbf r};\mathbb Z),
		\]
		which proves the homology statement.
		
		A pure shellable complex is homotopy equivalent to a wedge of spheres in its top dimension \cite[Section~11]{Bjorner95}. The dual has dimension $N_{\mathbf r}-3$, and the number of spheres is its top Betti number, namely $c-1$. This proves the homotopy statement when the topological dimension is at least zero. The Euler characteristic follows from the homology calculation.
	\end{proof}
	
	\begin{corollary}\label{cor:12}
		Under the hypotheses and notation of Theorem~\ref{thm:4}, the resolution of the ideal $I_{\Delta_{\mathbf r}^\vee}$ is pure if and only if
		\[
		n_1=\cdots=n_e=n
		\qquad\text{and}\qquad
		r_1=\cdots=r_{e-1}=r.
		\]
		In this case, it is linear if and only if $r=n-1$.
	\end{corollary}
	
	\begin{proof}
		Purity of the resolution in Theorem~\ref{thm:4} requires one common shift in each of its two free modules. Thus all $N_{\mathbf r}-n_m$ must be equal and all $N_{\mathbf r}-r_q$ must be equal, which is precisely the stated uniformity. Under this assumption, the resolution is linear if and only if
		\[
		(N_{\mathbf r}-r)-(N_{\mathbf r}-n)=1,
		\]
		that is, if and only if $r=n-1$.
	\end{proof}
	
	\section{Skeletons of the Alexander dual}\label{sec:6}
	
	We first simplify the finite alternating sum that occurs in the additional row of a skeleton resolution.
	
	\begin{lemma}\label{lem:2}
		Let $N,A,i,k$ be integers with $0\le A\le N$, $i\ge1$, and $k\ge-1$, and put $s=i+k+1\le N$. Then
		\[
		\sum_{q=0}^{k+1}(-1)^{k+1-q}
		\binom{N-q}{s-q}
		\binom{A}{N-q}
		=
		\binom{A}{N-s}
		\binom{A-N+s-1}{i-1}.
		\]
	\end{lemma}
	
	\begin{proof}
		Set $Q=N-s$. The identity
		\[
		\binom{A}{N-q}\binom{N-q}{s-q}
		=
		\binom AQ\binom{A-Q}{s-q}
		\]
		transforms the left side into
		\[
		\binom AQ\sum_{u=0}^{k+1}(-1)^u\binom{A-Q}{i+u},
		\qquad u=k+1-q.
		\]
		If $A\le Q$, both sides vanish under our binomial convention. Assume $A>Q$ and put $B=A-Q\ge1$. The finite alternating sum satisfies
		\[
		\sum_{u=0}^{k+1}(-1)^u\binom B{i+u}
		=
		\binom{B-1}{i-1}+(-1)^{k+1}\binom{B-1}{i+k+1}.
		\]
		Since $A\le N$, we have $B-1\le s-1=i+k$, so the last term is zero. The remaining term is
		\[
		\binom AQ\binom{B-1}{i-1}
		=
		\binom{A}{N-s}
		\binom{A-N+s-1}{i-1},
		\]
		as required.
	\end{proof}
	
	\begin{theorem}\label{thm:5}
		Let $\Delta_{\mathbf r}=\Delta_{\mathbf r}(n_1,\ldots,n_e)$ be the clique complex of a chordal graph with at least two maximal cliques $S_1,\ldots,S_e$ in a fixed leaf order. Put $n_m=|S_m|$, let $r_1,\ldots,r_{e-1}$ be the corresponding separator cardinalities.
		Let $-1\le k<N_{\mathbf r}-3$. The graded Betti numbers of $\SR{(\Delta_{\mathbf r}^\vee)_{(k)}}$ satisfy the following statements.
		\begin{enumerate}
			\item[(i)] If $j<k+1$, then
			\[
			\beta_{i,i+j}\bigl(\SR{(\Delta_{\mathbf r}^\vee)_{(k)}}\bigr)
			=
			\beta_{i,i+j}(\SR{\Delta_{\mathbf r}^\vee}).
			\]
			\item[(ii)] If $j>k+1$, then
			\[
			\beta_{i,i+j}\bigl(\SR{(\Delta_{\mathbf r}^\vee)_{(k)}}\bigr)=0.
			\]
			\item[(iii)] For every $i\ge1$, the row $j=i+k+1$ is
			\begin{align*}
				\beta_{i,i+k+1}\bigl(\SR{(\Delta_{\mathbf r}^\vee)_{(k)}}\bigr)
				&=
				\binom{N_{\mathbf r}}{i+k+1}\binom{i+k}{k+1}\\
				&\quad-
				\sum_{m=1}^e
				\binom{n_m}{N_{\mathbf r}-i-k-1}
				\binom{n_m-N_{\mathbf r}+i+k}{i-1}\\
				&\quad+
				\sum_{m=1}^{e-1}
				\binom{r_m}{N_{\mathbf r}-i-k-1}
				\binom{r_m-N_{\mathbf r}+i+k}{i-1}\\
				&\quad+
				\delta_{i,1}\bigl|\{m:N_{\mathbf r}-r_m=k+2\}\bigr|.
			\end{align*}
		\end{enumerate}
	\end{theorem}
	
	\begin{proof}
		For $k=-1$, the skeleton is $\{\varnothing\}$, and the statement reduces to the Koszul resolution of $\K$. Assume $k\ge0$. Parts~(i) and~(ii) are Proposition~\ref{prop:1}(i)--(ii). We calculate the remaining row.
		
		Put $s=i+k+1$. Proposition~\ref{prop:1}(iii) gives
		\begin{align*}
			\beta_{i,s}\bigl(\SR{(\Delta_{\mathbf r}^\vee)_{(k)}}\bigr)
			&=
			(-1)^i\sum_{q=0}^{k+1}(-1)^{s-q}
			\binom{N_{\mathbf r}-q}{s-q}f_{q-1}(\Delta_{\mathbf r}^\vee)\\
			&\quad-
			(-1)^i\sum_{i'>i}(-1)^{i'}
			\beta_{i',s}(\SR{\Delta_{\mathbf r}^\vee}).
		\end{align*}
		By Corollary~\ref{cor:9}, the second sum can be nonzero only when $i=1$ and $i'=2$. Its contribution is therefore
		\[
		\delta_{i,1}\beta_{2,k+2}(\SR{\Delta_{\mathbf r}^\vee})
		=
		\delta_{i,1}\bigl|\{m:N_{\mathbf r}-r_m=k+2\}\bigr|.
		\]
		
		For the first sum, use
		\[
		f_{q-1}(\Delta_{\mathbf r}^\vee)
		=
		\binom{N_{\mathbf r}}q-f_{N_{\mathbf r}-q-1}(\Delta_{\mathbf r}).
		\]
		The full simplex contributes
		\[
		\binom{N_{\mathbf r}}{i+k+1}\binom{i+k}{k+1},
		\]
		the standard Betti number of a simplex skeleton. Proposition~\ref{prop:2} writes the second term as a signed sum of binomial coefficients $\binom{A}{N_{\mathbf r}-q}$, with $A=n_m$ or $A=r_m$. Lemma~\ref{lem:2} evaluates each corresponding alternating sum. Substitution, with the signs from Proposition~\ref{prop:2}, gives part~(iii).
	\end{proof}
	
	The next corollary collects the algebraic and topological consequences.
	
	\begin{corollary}\label{cor:13}
		Under the hypotheses and notation of Theorem~\ref{thm:5}, for every $-1\le k<N_{\mathbf r}-3$, set
		\[
		\tau_k=
		\binom{N_{\mathbf r}-1}{k+1}
		-
		\sum_{m=1}^e\binom{n_m-1}{N_{\mathbf r}-k-2}
		+
		\sum_{m=1}^{e-1}\binom{r_m-1}{N_{\mathbf r}-k-2}
		\]
		and $A_k=\SR{(\Delta_{\mathbf r}^\vee)_{(k)}}$. Then $A_k$ is Cohen--Macaulay, and hence initially Cohen--Macaulay, with
		\[
		\dim A_k=k+1,
		\qquad
		\pdim A_k=N_{\mathbf r}-k-1,
		\qquad
		\operatorname{type}A_k=\tau_k.
		\]
		It is a level ring. For $k\ge0$, its $a$-invariant is $0$, the complex $(\Delta_{\mathbf r}^\vee)_{(k)}$ is shellable, and
		\[
		(\Delta_{\mathbf r}^\vee)_{(k)}
		\simeq
		\bigvee^{\tau_k}S^k.
		\]
		Consequently,
		\[
		\widetilde H_q((\Delta_{\mathbf r}^\vee)_{(k)};\mathbb Z)
		\cong
		\begin{cases}
			\mathbb Z^{\tau_k},&q=k,\\
			0,&q\ne k,
		\end{cases}
		\qquad(k\ge0),
		\]
		and the same ranks hold over every field.
	\end{corollary}
	
	\begin{proof}
		A skeleton of a Cohen--Macaulay complex is Cohen--Macaulay \cite[Exercise~5.1.23]{BH98}. Thus $A_k$ has dimension $k+1$ and projective dimension $N_{\mathbf r}-k-1$. Its type is the total Betti number in the last homological degree. Substitute $i=N_{\mathbf r}-k-1$ into Theorem~\ref{thm:5}(iii). The Kronecker term vanishes because $i\ge3$, and the remaining terms give $\tau_k$.
		
		At this last homological degree, the Betti rows inherited from $\SR{\Delta_{\mathbf r}^\vee}$ do not occur because that ring has projective dimension two. Hence the last free module in the resolution of $A_k$ is $R(-N_{\mathbf r})^{\tau_k}$. After dualizing with respect to $R(-N_{\mathbf r})$, the canonical module is generated in degree zero. Thus $A_k$ is level. This levelness is also a special case of Hibi's general level theorem for proper skeletons of Cohen--Macaulay complexes \cite{HibiLevel}; the present calculation additionally identifies the last shift and the type explicitly. When $k\ge0$, the equality $a(A_k)=\reg A_k-\dim A_k=0$ also follows from Corollary~\ref{cor:15} below.
		
		Corollary~\ref{cor:11} shows that $\Delta_{\mathbf r}^\vee$ is shellable. Every proper pure skeleton of a shellable complex is shellable \cite[Lemma~1.6]{HHZ04}. A pure shellable $k$-dimensional complex is homotopy equivalent to a wedge of $k$-spheres \cite[Section~11]{Bjorner95}. Hochster's formula on the full vertex set identifies the number of spheres with the last Betti number $\beta_{N_{\mathbf r}-k-1,N_{\mathbf r}}(A_k)=\tau_k$. The integral homology statement follows from this homotopy equivalence.
		
	\end{proof}
	
	\begin{corollary}\label{cor:14}
		With $\tau_k$ and $A_k$ as in Corollary~\ref{cor:13}, let $G$ be the chordal graph with $\Delta_{\mathbf r}=\operatorname{Cl}(G)$. For $0\le k\le N_{\mathbf r}-5$,
		\[
		\tau_k\ge\binom{N_{\mathbf r}-3}{k+1}\ge2.
		\]
		If $N_{\mathbf r}\ge4$ and $c$ is the number of connected components of $G$, then at the last proper skeleton
		\[
		\tau_{N_{\mathbf r}-4}=|E(\overline G)|-c+1.
		\]
		For $k\ge0$, the following conditions are equivalent:
		\begin{enumerate}
			\item[(i)] $A_k$ is Gorenstein;
			\item[(ii)] $\tau_k=1$;
			\item[(iii)] $k=N_{\mathbf r}-4$, $e=2$, $n_1=n_2=N_{\mathbf r}-1$, and $r_1=N_{\mathbf r}-2$;
			\item[(iv)] $\Delta_{\mathbf r}^\vee$ is a simplex on $N_{\mathbf r}-2$ vertices, with the remaining two vertices ghost, and $(\Delta_{\mathbf r}^\vee)_{(k)}$ is the boundary of that simplex;
			\item[(v)] $k=N_{\mathbf r}-4$ and $G=K_{N_{\mathbf r}}\setminus\{uv\}$ for some distinct vertices $u,v$.
		\end{enumerate}
		In every other case with $k\ge0$, $\tau_k\ge2$.
	\end{corollary}
	
	\begin{proof}
		Because $e\ge2$, the chordal graph $G$ is not complete. Choose a nonedge $\{u,v\}$ of $G$. Its complement $X\setminus\{u,v\}$ is a facet of $\Delta_{\mathbf r}^\vee$ with $N_{\mathbf r}-2$ vertices. If $0\le k\le N_{\mathbf r}-5$, the $k$-skeleton of this facet is contained in $(\Delta_{\mathbf r}^\vee)_{(k)}$ and has top reduced homology rank
		\[
		\binom{N_{\mathbf r}-3}{k+1}.
		\]
		Since a $k$-dimensional complex has no boundaries coming from dimension $k+1$, these independent top cycles remain independent in the full $k$-skeleton. Hence
		\[
		\tau_k\ge\binom{N_{\mathbf r}-3}{k+1}\ge2.
		\]
		Thus a Gorenstein proper skeleton with $k\ge0$ can occur only when $k=N_{\mathbf r}-4$.
		
		Set $L=(\Delta_{\mathbf r}^\vee)_{(N_{\mathbf r}-4)}$. The facets of $\Delta_{\mathbf r}^\vee$ are the complements of the nonedges of $G$, so
		\[
		f_{N_{\mathbf r}-3}(\Delta_{\mathbf r}^\vee)=|E(\overline G)|.
		\]
		By Corollary~\ref{cor:11},
		\[
		\operatorname{rank}\widetilde H_{N_{\mathbf r}-3}(\Delta_{\mathbf r}^\vee;\mathbb Z)=c-1,
		\qquad
		\widetilde H_{N_{\mathbf r}-4}(\Delta_{\mathbf r}^\vee;\mathbb Z)=0.
		\]
		In the simplicial chain complex, it follows that
		\[
		\operatorname{rank}\widetilde H_{N_{\mathbf r}-4}(L;\mathbb Z)
		=\operatorname{rank}\operatorname{im}\partial_{N_{\mathbf r}-3}
		=|E(\overline G)|-(c-1).
		\]
		The left side is $\tau_{N_{\mathbf r}-4}$, which proves the formula for the last proper skeleton.
		
		Since $A_k$ is Cohen--Macaulay, it is Gorenstein exactly when its type is one, so conditions~(i) and~(ii) are equivalent. If $G$ is disconnected and $N_{\mathbf r}\ge4$, the pairs of vertices in distinct connected components alone give at least $c+1$ nonedges; hence the formula for the last proper skeleton gives $\tau_{N_{\mathbf r}-4}\ge2$. Therefore type one implies $c=1$, and then the formula for the last proper skeleton gives $|E(\overline G)|=1$. Thus $G=K_{N_{\mathbf r}}\setminus\{uv\}$ for a unique missing edge $uv$. Its maximal cliques are $X\setminus\{u\}$ and $X\setminus\{v\}$, so $e=2$, $n_1=n_2=N_{\mathbf r}-1$, and $r_1=N_{\mathbf r}-2$. Conversely, these numerical conditions give precisely a complete graph with one missing edge. Its Alexander dual is the simplex on $X\setminus\{u,v\}$ with $u$ and $v$ ghost, and the $(N_{\mathbf r}-4)$-skeleton is the boundary of that simplex. Its Stanley--Reisner ring is Gorenstein. This proves the equivalence of~(i)--(v) and the final assertion.
	\end{proof}
	
	\begin{corollary}\label{cor:15}
		Under the hypotheses and notation of Theorem~\ref{thm:5}, for $-1\le k<N_{\mathbf r}-3$,
		\[
		\reg\SR{(\Delta_{\mathbf r}^\vee)_{(k)}}=k+1,
		\qquad
		\reg I_{(\Delta_{\mathbf r}^\vee)_{(k)}}=k+2.
		\]
		If $r_{\min}\ge1$ and $N_{\mathbf r}-r_{\min}-3\le k<N_{\mathbf r}-3$, then
		\[
		\reg I_{\Delta_{\mathbf r}^\vee}
		\le
		\reg I_{(\Delta_{\mathbf r}^\vee)_{(k)}}=k+2,
		\]
		with equality at $k=N_{\mathbf r}-r_{\min}-3$. If $r_{\min}=0$, no proper skeleton lies in this comparison range; in fact,
		\[
		\reg I_{(\Delta_{\mathbf r}^\vee)_{(k)}}
		<
		\reg I_{\Delta_{\mathbf r}^\vee}
		\qquad(-1\le k<N_{\mathbf r}-3).
		\]
	\end{corollary}
	
	\begin{proof}
		Theorem~\ref{thm:5} has no row of degree shift greater than $k+1$, and Corollary~\ref{cor:13} shows that its last Betti number in that row is nonzero. Hence the quotient regularity is $k+1$, and the regularity of its nonzero Stanley--Reisner ideal is $k+2$.
		
		Corollary~\ref{cor:9} gives $\reg I_{\Delta_{\mathbf r}^\vee}=N_{\mathbf r}-r_{\min}-1$. The stated comparison is therefore equivalent to $k\ge N_{\mathbf r}-r_{\min}-3$. This interval contains a proper skeleton exactly when $r_{\min}\ge1$, and equality holds at its lower endpoint. If $r_{\min}=0$, then every proper skeleton satisfies $k+2\le N_{\mathbf r}-2<N_{\mathbf r}-1=\reg I_{\Delta_{\mathbf r}^\vee}$.
	\end{proof}
	
	We next compare the dual skeleton with the skeleton of the full simplex.
	
	\begin{theorem}\label{thm:6}
		Let $\Delta_{\mathbf r}=\Delta_{\mathbf r}(n_1,\ldots,n_e)$ be the clique complex of a chordal graph on vertex set $X$, with at least two maximal cliques $S_1,\ldots,S_e$ in a fixed leaf order. Put $n_m=|S_m|$, let $r_1,\ldots,r_{e-1}$ be the corresponding separator cardinalities.
		Let $\boldsymbol\Delta_{N_{\mathbf r}}$ denote the simplex on $X$, and let $-1\le k<N_{\mathbf r}-3$. Then
		\[
		(\Delta_{\mathbf r}^\vee)_{(k)}
		=(\boldsymbol\Delta_{N_{\mathbf r}})_{(k)}
		\quad\Longleftrightarrow\quad
		k\le N_{\mathbf r}-n_{\max}-2.
		\]
		In this range,
		\[
		\beta_{i,i+k+1}\bigl(\SR{(\Delta_{\mathbf r}^\vee)_{(k)}}\bigr)
		=
		\binom{N_{\mathbf r}}{i+k+1}\binom{i+k}{k+1}
		\qquad(i\ge1).
		\]
		If $k\ge0$, then
		\[
		(\Delta_{\mathbf r}^\vee)_{(k)}
		\simeq
		\bigvee^{\binom{N_{\mathbf r}-1}{k+1}}S^k.
		\]
	\end{theorem}
	
	\begin{proof}
		The minimal nonfaces of $\Delta_{\mathbf r}^\vee$ are the complements $X\setminus S_m$. The smallest of them has cardinality $N_{\mathbf r}-n_{\max}$. Hence every subset of $X$ of cardinality at most $k+1$ is a face of $\Delta_{\mathbf r}^\vee$ exactly when
		\[
		k+1<N_{\mathbf r}-n_{\max},
		\]
		which is the stated inequality.
		
		More explicitly, if $k\le N_{\mathbf r}-n_{\max}-2$, neither complex has a nonface of cardinality at most $k+1$, so the two $k$-skeletons agree. Conversely, if $k\ge N_{\mathbf r}-n_{\max}-1$, choose a facet $S_m$ of cardinality $n_{\max}$. Its complement has cardinality at most $k+1$ and is a minimal nonface of the dual. It is therefore a face of the simplex skeleton but not of the dual skeleton.
		
		The Betti formula is the standard formula for a simplex skeleton. The homotopy statement follows either from that formula and Corollary~\ref{cor:13}, or from the standard shelling of a simplex skeleton \cite[Section~11]{Bjorner95}.
	\end{proof}
	
	\begin{corollary}\label{cor:16}
		Under the hypotheses and notation of Theorem~\ref{thm:6}, for $-1\le k<N_{\mathbf r}-3$, the Stanley--Reisner ideal $I_{(\Delta_{\mathbf r}^\vee)_{(k)}}$ has a $(k+2)$-linear resolution if and only if
		\[
		k\le N_{\mathbf r}-n_{\max}-2.
		\]
		Equivalently, this happens exactly when $(\Delta_{\mathbf r}^\vee)_{(k)}$ is the $k$-skeleton of the full simplex on $X$.
	\end{corollary}
	
	\begin{proof}
		In the range of Theorem~\ref{thm:6}, the ideal is the squarefree Veronese ideal generated by all squarefree monomials of degree $k+2$, so it has a $(k+2)$-linear resolution.
		
		Suppose instead that $k\ge N_{\mathbf r}-n_{\max}-1$. Then $I_{(\Delta_{\mathbf r}^\vee)_{(k)}}$ has a minimal generator $x_{X\setminus S_m}$ of degree $N_{\mathbf r}-n_{\max}\le k+1$. Because $k<N_{\mathbf r}-3=\dim\Delta_{\mathbf r}^\vee$, the dual contains a face of cardinality $k+2$; its monomial becomes a minimal nonface of the $k$-skeleton and gives a minimal generator of degree $k+2$. Thus the ideal is not generated in one degree and cannot have a linear resolution.
	\end{proof}
	
	The next result determines precisely which numerical data are determined by the graded Betti table of the dual when the number of vertices is allowed to vary.
	
	\begin{theorem}\label{thm:7}
		Let $\Delta$ and $\Gamma$ be chordal clique complexes on $N$ and $N'$ vertices, respectively, and assume that each has at least two facets. Choose leaf orders, with facet and separator sizes
		\[
		\{n_1,\ldots,n_e\},\quad \{r_1,\ldots,r_{e-1}\}
		\]
		for $\Delta$, and
		\[
		\{n'_1,\ldots,n'_{e'}\},\quad \{r'_1,\ldots,r'_{e'-1}\}
		\]
		for $\Gamma$. Put $\delta=N'-N$. Then the following conditions are equivalent.
		\begin{enumerate}
			\item[(i)] The rings $\SR{\Delta^\vee}$ and $\SR{\Gamma^\vee}$ have the same graded Betti table.
			\item[(ii)] The multisets of facet and separator cardinalities differ by the common shift $\delta$:
			\[
			\{n'_1,\ldots,n'_{e'}\}=\{n_1+\delta,\ldots,n_e+\delta\},
			\]
			\[
			\{r'_1,\ldots,r'_{e'-1}\}=\{r_1+\delta,\ldots,r_{e-1}+\delta\}.
			\]
		\end{enumerate}
		In particular, if $N=N'$, the graded Betti table determines both multisets exactly. In that case, $\SR{\Delta_{(k)}}$ and $\SR{\Gamma_{(k)}}$ have the same graded Betti table for every $k$, and the same is true of $\SR{(\Delta^\vee)_{(k)}}$ and $\SR{(\Gamma^\vee)_{(k)}}$ whenever the corresponding skeletons are defined.
	\end{theorem}
	
	\begin{proof}
		Assume first that the dual rings have the same graded Betti table. The total Betti number in homological degree one gives $e=e'$. By Corollary~\ref{cor:9}, their Betti numbers in homological degree one identify the same multiset of generator degrees:
		\[
		\{N-n_m:1\le m\le e\}
		=
		\{N'-n'_m:1\le m\le e\}.
		\]
		Since $N'=N+\delta$, subtracting these degrees from the respective vertex counts gives
		\[
		\{n'_m:1\le m\le e\}=\{n_m+\delta:1\le m\le e\}.
		\]
		The Betti numbers in homological degree two give, in the same way,
		\[
		\{N-r_q:1\le q\le e-1\}
		=
		\{N'-r'_q:1\le q\le e-1\},
		\]
		and hence $\{r'_q\}=\{r_q+\delta\}$. This proves~(ii).
		
		Conversely, the common shift equalities imply
		\[
		\{N'-n'_m\}=\{N-n_m\},
		\qquad
		\{N'-r'_q\}=\{N-r_q\}.
		\]
		Corollary~\ref{cor:9} therefore gives identical graded Betti tables for the two dual rings. If $N=N'$, then $\delta=0$ and the numerical data agree exactly. The final assertions then follow from Theorems~\ref{thm:1} and~\ref{thm:5}.
	\end{proof}
	
	\begin{remark}\label{rem:5}
		Theorem~\ref{thm:7} determines only the facet and separator cardinalities, not the clique tree. Example~\ref{ex:1} gives nonisomorphic complexes with the same facet and separator sizes. Corollary~\ref{cor:7} gives finer information: the multigraded Betti numbers determine the actual facet and separator subsets on a fixed labeled vertex set. A particular differential written as in Theorem~\ref{thm:4} also determines the chosen branches, but a minimal resolution is unique only up to graded changes of basis, so the chosen relation tree is not itself a canonical invariant of the abstract graded module.
		
		The common shift in Theorem~\ref{thm:7} is exactly the numerical effect of coning. Adding one new vertex to every facet increases $N$, every $n_m$, and every $r_q$ by one, while preserving the differences $N-n_m$ and $N-r_q$. The theorem shows that, when the number of vertices is not specified, the common shift is the only part of the facet and separator cardinality data not determined by the graded Betti table. It does not assert that two complexes with shifted data are isomorphic after coning, because the graded table does not determine the clique tree.
	\end{remark}
	
	\begin{corollary}\label{cor:17}
		Let $\Delta_{\mathbf r}=\Delta_{\mathbf r}(n_1,\ldots,n_e)$ be a chordal clique complex. All graded Betti numbers computed in Theorems~\ref{thm:1} and~\ref{thm:5} and Corollary~\ref{cor:9} are independent of the characteristic of $\K$. The projective dimensions, regularities, Cohen--Macaulay types, and reduced homology dimensions obtained from these formulas are therefore also characteristic independent.
	\end{corollary}
	
	\begin{proof}
		Each formula is an integer expression involving only $N_{\mathbf r}$ and the facet and separator cardinalities. No term depends on the ground field.
	\end{proof}
	
	\section{Examples}\label{sec:7}
	
	The first example shows that the numerical data determine the graded Betti tables in this paper but not the clique tree.
	
	\begin{example}\label{ex:1}
		Let $\Delta_P$ have facets
		\[
		\{a,b,c\},\quad \{c,d,e\},\quad \{e,f,g\},\quad \{g,h,i\},
		\]
		and let $\Delta_S$ have facets
		\[
		\{a,b,c\},\quad \{a,d,e\},\quad \{b,f,g\},\quad \{c,h,i\}.
		\]
		Both complexes have $N=9$, facet size multiset $\{3,3,3,3\}$, and separator size multiset $\{1,1,1\}$. The facet intersection graph of $\Delta_P$ is the path $P_4$, whereas that of $\Delta_S$ is the star $K_{1,3}$; in particular, the complexes are not isomorphic. Nevertheless, Theorem~\ref{thm:7} shows that their Alexander duals and all corresponding skeletons considered in this paper have the same graded Betti tables. Their multigraded Betti numbers differ, because Corollary~\ref{cor:7} determines the actual facet and separator subsets.
	\end{example}
	
	We next apply the formulas to a nonuniform complex.
	
	\begin{example}\label{ex:2}
		Let $\Delta_{\mathbf r}$ have facets
		\[
		S_1=\{x_1,x_2,x_3\},\qquad
		S_2=\{x_2,x_3,x_4,x_5,x_6\},\qquad
		S_3=\{x_4,x_5,x_6,x_7,x_8,x_9\}.
		\]
		Then
		\[
		(n_1,n_2,n_3)=(3,5,6),
		\qquad
		(r_1,r_2)=(2,3),
		\qquad
		N_{\mathbf r}=9,
		\]
		with branches $p(2)=1$ and $p(3)=2$.
		
		By Proposition~\ref{prop:2},
		\[
		f(\Delta_{\mathbf r})=(1,9,24,30,20,7,1).
		\]
		The linear row inherited by every proper skeleton of positive dimension is, by Theorem~\ref{thm:1}(c),
		\[
		(\beta_{1,2},\beta_{2,3},\beta_{3,4},\beta_{4,5},\beta_{5,6},\beta_{6,7})
		=(12,30,34,21,7,1).
		\]
		For the $2$-skeleton, the additional row $j-i=3$ is
		\[
		(\beta_{1,4},\beta_{2,5},\beta_{3,6},\beta_{4,7},\beta_{5,8},\beta_{6,9})
		=(20,93,173,161,75,14).
		\]
		Thus the Betti table of $\SR{(\Delta_{\mathbf r})_{(2)}}$ is
		\[
		\begin{array}{r|rrrrrrr}
			&0&1&2&3&4&5&6\\ \hline
			\mathrm{total}&1&32&123&207&182&82&15\\
			0&1&.&.&.&.&.&.\\
			1&.&12&30&34&21&7&1\\
			2&.&.&.&.&.&.&.\\
			3&.&20&93&173&161&75&14
		\end{array}
		\]
		The $0$-skeleton is instead given by Theorem~\ref{thm:1}(b): $\beta_{i,i+1}=i\binom9{i+1}$. For $k=2=r_{\min}$, Theorem~\ref{thm:3} gives
		\[
		\lambda_2=
		\binom23+\binom43+\binom53
		-\binom13-\binom23=14,
		\qquad
		\mu=1.
		\]
		Thus $\widetilde H_2((\Delta_{\mathbf r})_{(2)};\mathbb Z)\cong\mathbb Z^{14}$, the last nonzero Betti numbers in the two rows are $\beta_{6,9}=14$ and $\beta_{6,7}=1$, and the Cohen--Macaulay type is $15$. Since the last free module has two shifts, this skeleton is not level.
		
		For the Alexander dual, the minimal generators are
		\[
		g_1=x_4x_5x_6x_7x_8x_9,
		\qquad
		g_2=x_1x_7x_8x_9,
		\qquad
		g_3=x_1x_2x_3.
		\]
		Theorem~\ref{thm:4} gives the minimal resolution
		\[
		0\longrightarrow R(-7)\oplus R(-6)
		\xrightarrow{\partial_2}
		R(-6)\oplus R(-4)\oplus R(-3)
		\xrightarrow{\partial_1}
		I_{\Delta_{\mathbf r}^\vee}
		\longrightarrow0,
		\]
		where
		\[
		\partial_1=\begin{bmatrix}g_1&g_2&g_3\end{bmatrix},
		\qquad
		\partial_2=
		\begin{bmatrix}
			x_1&0\\
			-x_4x_5x_6&x_2x_3\\
			0&-x_7x_8x_9
		\end{bmatrix}.
		\]
		Internal degree $6$ occurs among both the minimal generators and the first syzygies. Their contributions cancel in the Hilbert numerator, which explains why that numerator cannot determine the two Betti rows separately. The direct resolution separates these two contributions.
		
		Here $r_{\min}=2$, so Corollary~\ref{cor:9} gives
		\[
		\reg\SR{\Delta_{\mathbf r}^\vee}=9-2-2=5.
		\]
		Both separators are nonempty, so $\Delta_{\mathbf r}$ is connected. Corollary~\ref{cor:11} therefore shows that $\Delta_{\mathbf r}^\vee$ is contractible.
		
		The bound in Theorem~\ref{thm:6} is $N_{\mathbf r}-n_{\max}-2=1$. Hence $(\Delta_{\mathbf r}^\vee)_{(1)}$ is the $1$-skeleton of the $8$-simplex and
		\[
		\beta_{i,i+2}=\binom9{i+2}\binom{i+1}2.
		\]
		At $k=2$, the skeleton is no longer a simplex skeleton. Corollary~\ref{cor:13} gives $\tau_2=55$, and therefore
		\[
		(\Delta_{\mathbf r}^\vee)_{(2)}\simeq\bigvee^{55}S^2.
		\]
		Its Stanley--Reisner ring is Cohen--Macaulay and level of type $55$.
	\end{example}
	
	\section{A binomial convolution from the Hilbert series}\label{sec:8}
	
	The Hilbert series of $\SR{\Delta_{\mathbf r}}$ has two useful expressions. Proposition~\ref{prop:2} writes it in terms of the facet and separator cardinalities, while the standard $f$-vector expression writes it in terms of the face numbers. Comparing the numerators of these expressions gives the following consequence. The underlying convolution is classical; the point here is that the facet and separator data organize it in a form adapted to chordal clique complexes. When $r_1=\cdots=r_{e-1}=1$, the formula specializes to the single-vertex separator case of thick trees considered in~\cite{Fr2}.
	
	\begin{corollary}\label{cor:18}
		Let $\Delta_{\mathbf r}=\Delta_{\mathbf r}(n_1,\ldots,n_e)$ be the chordal clique complex, with separator cardinalities $r_1,\ldots,r_{e-1}$ and $N_{\mathbf r}$ vertices. For every integer $j\geq0$,
		\begin{align*}
			&\sum_{m=1}^{e}\binom{N_{\mathbf r}-n_m}{j}
			-\sum_{m=1}^{e-1}\binom{N_{\mathbf r}-r_m}{j} \\
			&\qquad=
			\sum_{q=0}^{j}(-1)^q f_{q-1}(\Delta_{\mathbf r})
			\binom{N_{\mathbf r}-q}{j-q} \\
			&\qquad=
			\sum_{q=0}^{j}(-1)^q
			\left(
			\sum_{m=1}^{e}\binom{n_m}{q}
			-\sum_{m=1}^{e-1}\binom{r_m}{q}
			\right)
			\binom{N_{\mathbf r}-q}{j-q}.
		\end{align*}
		In particular, if $j=i+1$ with $i\geq1$, then each expression is equal to
		\[
		-\beta_{i,i+1}\bigl(\SR{\Delta_{\mathbf r}}\bigr).
		\]
	\end{corollary}
	
	\begin{proof}
		By Proposition~\ref{prop:2}, writing the Hilbert series over the denominator $(1-z)^{N_{\mathbf r}}$ gives the numerator
		\[
		P(z)=
		\sum_{m=1}^{e}(1-z)^{N_{\mathbf r}-n_m}
		-\sum_{m=1}^{e-1}(1-z)^{N_{\mathbf r}-r_m}.
		\]
		On the other hand, the standard $f$-vector expression for the same Hilbert series gives
		\[
		P(z)=
		\sum_{q\geq0}
		f_{q-1}(\Delta_{\mathbf r})z^q
		(1-z)^{N_{\mathbf r}-q}.
		\]
		The coefficient of $z^j$ in the first expression is
		\[
		(-1)^j
		\left(
		\sum_{m=1}^{e}\binom{N_{\mathbf r}-n_m}{j}
		-\sum_{m=1}^{e-1}\binom{N_{\mathbf r}-r_m}{j}
		\right),
		\]
		whereas its coefficient in the second expression is
		\[
		\sum_{q=0}^{j}
		(-1)^{j-q}f_{q-1}(\Delta_{\mathbf r})
		\binom{N_{\mathbf r}-q}{j-q}.
		\]
		Equating these coefficients and multiplying by $(-1)^j$ gives the first equality. The second follows from the face number formula in Proposition~\ref{prop:2}. Finally, the assertion for $j=i+1$ follows from the formula for the linear row in Theorem~\ref{thm:1}.
	\end{proof}
	
	The preceding identity at the actual vertex count has a useful polynomial extension. We first state the classical convolution that underlies it.
	
	\begin{lemma}\label{lem:3}
		For $j\ge0$, interpret
		\[
		\binom{x}{j}=\frac{x(x-1)\cdots(x-j+1)}{j!},
		\qquad \binom{x}{0}=1,
		\]
		as a polynomial in $x$. If $A,j$ are nonnegative integers, then
		\[
		\binom{x-A}{j}
		=
		\sum_{q=0}^{j}(-1)^q\binom{A}{q}\binom{x-q}{j-q}.
		\]
	\end{lemma}
	
	\begin{proof}
		For every integer $x\ge A$, the identity is Lemma~\ref{lem:1} with $n=x$ and $s=j$. Both sides are polynomials in $x$ of degree at most $j$. Since they agree for infinitely many integers $x$, they agree identically.
	\end{proof}
	
	\begin{theorem}\label{thm:8}
		Let $\Delta_{\mathbf r}=\Delta_{\mathbf r}(n_1,\ldots,n_e)$ be the clique complex of a chordal graph with at least two maximal cliques in a fixed leaf order, with facet cardinalities $n_1,\ldots,n_e$ and separator cardinalities $r_1,\ldots,r_{e-1}$.
		For every integer $j\geq0$, one has the polynomial identity
		\begin{align*}
			&\sum_{m=1}^{e}\binom{x-n_m}{j}
			-\sum_{m=1}^{e-1}\binom{x-r_m}{j} \\
			&\qquad=
			\sum_{q=0}^{j}(-1)^q
			\left(
			\sum_{m=1}^{e}\binom{n_m}{q}
			-\sum_{m=1}^{e-1}\binom{r_m}{q}
			\right)
			\binom{x-q}{j-q}.
		\end{align*}
		At $x=N_{\mathbf r}$, this specializes to Corollary~\ref{cor:18}.
	\end{theorem}
	
	\begin{proof}
		Apply Lemma~\ref{lem:3} with $A=n_m$ and with $A=r_m$, and sum with the indicated signs. The specialization $x=N_{\mathbf r}$ follows from Proposition~\ref{prop:2} and Corollary~\ref{cor:18}.
	\end{proof}
	
	\begin{corollary}\label{cor:19}
		Let $\Delta_{\mathbf r}=\Delta_{\mathbf r}(n,\ldots,n)$ be a chordal clique complex with $e$ facets, and assume that every separator has cardinality $r$. Then, for every $j\geq0$,
		\begin{align*}
			e\binom{x-n}{j}-(e-1)\binom{x-r}{j}
			&=
			\sum_{q=0}^{j}(-1)^q
			\left(
			e\binom{n}{q}-(e-1)\binom{r}{q}
			\right)
			\binom{x-q}{j-q}.
		\end{align*}
		For the corresponding uniform chordal clique complex,
		$x=N_{\mathbf r}=en-(e-1)r$ gives the specialization at the actual vertex count.
	\end{corollary}
	
	\begin{proof}
		This is Theorem~\ref{thm:8} after substituting $n_m=n$ and $r_m=r$.
	\end{proof}
	
	\begin{corollary}\label{cor:20}
		For nonnegative integers $n,r$ and every $j\geq0$,
		\[
		\binom{r-n}{j}
		=
		\sum_{q=0}^{j}(-1)^q\binom{n}{q}\binom{r-q}{j-q},
		\]
		where a binomial coefficient with a possibly negative upper entry is interpreted through the binomial polynomial above.
	\end{corollary}
	
	\begin{proof}
		Set $A=n$ and $x=r$ in Lemma~\ref{lem:3}. This is a specialization of the polynomial identity and is independent of the vertex count relation defining $N_{\mathbf r}$.
	\end{proof}
	
	\begin{remark}\label{rem:6}
		Lemma~\ref{lem:3} and Corollary~\ref{cor:20} are classical forms of the alternating Chu--Vandermonde convolution. The contribution here is not a new underlying binomial identity. Rather, Corollary~\ref{cor:18} and Theorem~\ref{thm:8} show how that convolution is organized naturally by the facet and separator data of a chordal clique complex. Corollary~\ref{cor:19} is the immediate uniform specialization.
	\end{remark}
	
	\section*{Acknowledgments}
	The author thanks Ralf Fr\"oberg for encouraging correspondence, careful comments on early drafts, and suggestions that improved the presentation.


\begin{thebibliography}{99}
		
		\bibitem{AB57}
		M. Auslander and D. A. Buchsbaum, \emph{Homological dimension in local rings}, Trans. Amer. Math. Soc. \textbf{85} (1957), 390--405, \doi{10.1090/S0002-9947-1957-0086822-7}.
		
		\bibitem{BlairPeyton}
		J. R. S. Blair and B. W. Peyton, \emph{An introduction to chordal graphs and clique trees}, in A. George, J. R. Gilbert and J. W. H. Liu (eds.), Graph Theory and Sparse Matrix Computation, IMA Volumes in Mathematics and its Applications, vol. 56, Springer, New York, 1993, pp. 1--29, \doi{10.1007/978-1-4613-8369-7_1}.
		
		\bibitem{Bjorner95}
		A. Bj\"orner, \emph{Topological methods}, in R. L. Graham, M. Gr\"otschel and L. Lov\'asz (eds.), Handbook of Combinatorics, vol. 2, Elsevier, Amsterdam, 1995, pp. 1819--1872.
		
		\bibitem{BjornerTancer09}
		A. Bj\"orner and M. Tancer, \emph{Combinatorial Alexander duality---a short and elementary proof}, Discrete Comput. Geom. \textbf{42} (2009), 586--593, \doi{10.1007/s00454-008-9102-x}.
		
		\bibitem{BH98}
		W. Bruns and J. Herzog, \emph{Cohen--Macaulay Rings}, revised ed., Cambridge Studies in Advanced Mathematics, vol. 39, Cambridge University Press, Cambridge, 1998, \doi{10.1017/CBO9780511608681}.
		
		\bibitem{ER98}
		J. A. Eagon and V. Reiner, \emph{Resolutions of Stanley--Reisner rings and Alexander duality}, J. Pure Appl. Algebra \textbf{130} (1998), 265--275, \doi{10.1016/S0022-4049(97)00097-2}.
		
		\bibitem{FaridiHersey}
		S. Faridi and B. Hersey, \emph{Resolutions of monomial ideals of projective dimension $1$}, Comm. Algebra \textbf{45} (2017), 5453--5464, \doi{10.1080/00927872.2017.1313422}.
		
		\bibitem{Fr}
		R. Fr\"oberg, \emph{On Stanley--Reisner rings}, in Topics in Algebra, Part 2 (Warsaw, 1988), Banach Center Publ. \textbf{26}, PWN, Warsaw, 1990, pp. 57--70, \doi{10.4064/-26-2-57-70}.
		
		\bibitem{Fr1}
		R. Fr\"oberg, \emph{Betti numbers of fat forests and their Alexander dual}, J. Algebraic Combin. \textbf{56} (2022), 1023--1030, \doi{10.1007/s10801-022-01143-0}.
		
		\bibitem{Fr2}
		R. Fr\"oberg, \emph{Betti numbers of thick trees}, Util. Math. \textbf{127} (2026), 173--181, \doi{10.61091/um127-12}.
		
		\bibitem{HHMTZ08}
		J. Herzog, T. Hibi, S. Murai, N. V. Trung and X. Zheng, \emph{Kruskal--Katona type theorems for clique complexes arising from chordal and strongly chordal graphs}, Combinatorica \textbf{28} (2008), 315--323, \doi{10.1007/s00493-008-2319-8}.
		
		\bibitem{HHZ04}
		J. Herzog, T. Hibi and X. Zheng, \emph{Dirac's theorem on chordal graphs and Alexander duality}, European J. Combin. \textbf{25} (2004), 949--960, \doi{10.1016/j.ejc.2003.12.008}.
		
		\bibitem{HHZPowers}
		J. Herzog, T. Hibi and X. Zheng, \emph{Monomial ideals whose powers have a linear resolution}, Math. Scand. \textbf{95} (2004), 23--32, \doi{10.7146/math.scand.a-14446}.
		
		\bibitem{HerzogHibi99}
		J. Herzog and T. Hibi, \emph{Componentwise linear ideals}, Nagoya Math. J. \textbf{153} (1999), 141--153, \doi{10.1017/S0027763000006930}.
		
		\bibitem{HibiLevel}
		T. Hibi, \emph{Level rings and algebras with straightening laws}, J. Algebra \textbf{117} (1988), 343--362, \doi{10.1016/0021-8693(88)90111-1}.
		
		\bibitem{Ho77}
		M. Hochster, \emph{Cohen--Macaulay rings, combinatorics, and simplicial complexes}, in Ring Theory II, Lecture Notes in Pure and Appl. Math. \textbf{26}, Dekker, New York, 1977, pp. 171--223.
		
		\bibitem{Namiq1}
		M. R. Namiq, \emph{Initially Cohen--Macaulay modules}, New Math. Nat. Comput. (2026), published online, \doi{10.1142/S1793005728500226}.
		
		\bibitem{Namiq2}
		M. R. Namiq, \emph{The graded Betti numbers of the skeletons of simplicial complexes}, Indag. Math. (N.S.) (2026), in press, \doi{10.1016/j.indag.2026.03.006}.
		
		\bibitem{PPTZ26}
		M. R. Pournaki, M. Poursoltani, N. Terai and S. Yassemi, \emph{Level simplicial complexes of codimension two and a conjecture of Stanley on pure $O$-sequences}, Comm. Algebra (2026), published online 27 June 2026, \doi{10.1080/00927872.2026.2664042}.
		
		\bibitem{RV15}
		J. Roksvold and H. Verdure, \emph{Betti numbers of skeletons}, arXiv:1502.05670 (2015), \doi{10.48550/arXiv.1502.05670}.
		
	\end{thebibliography}
\end{document}